\documentclass[reqno]{amsart}

\usepackage[utf8]{inputenc}
\usepackage[T1]{fontenc}
\usepackage[english]{babel}
\usepackage{amsmath, amssymb, latexsym, bm}
\usepackage{hyperref}

\usepackage{graphicx}
\usepackage{subcaption} 
\usepackage{pgfplots}
\usepackage{tikz}
\usetikzlibrary{tikzmark, calc}

\usepackage{fancyhdr}
\usepackage{float}
\usepackage{array, multirow, makecell}
\usepackage{color}

\definecolor{dkgreen}{rgb}{0,0.6,0}
\definecolor{gray}{rgb}{0.5,0.5,0.5}
\definecolor{mauve}{rgb}{0.58,0,0.82}

\hypersetup{
    colorlinks,%
    citecolor=red,%
    filecolor=yellow,%
    linkcolor=blue,%
    urlcolor=black
}

\newtheorem{theorem}{Theorem}[section]
\newtheorem{lemma}[theorem]{Lemma}
\newtheorem{proposition}[theorem]{Proposition}
\newtheorem{definition}[theorem]{Definition}

\usepackage[toc,page]{appendix}
\begin{document}
\title
[\hfil A PDE model for stress propagation between two interconnected zones]{An advection-diffusion-reaction model for stress propagation between two interconnected zones}
\author[K. Khalil, I. L. Mikiela Ndzoumbou]
{K. Khalil$^{1,*}$, I. L. Mikiela Ndzoumbou$^1$}
  
\address{Kamal Khalil,\newline
LMAH, University of Le Havre Normandie, FR-CNRS-3335, ISCN, Le Havre 76600, France.}
\email{kamal.khalil.00@gmail.com}

\address{ Irmand Leblond MIKIELA NDZOUMBOU,\newline
LMAH, University of Le Havre Normandie, FR-CNRS-3335, ISCN, Le Havre 76600, France.}
\email{irmandmikiela2016@gmail.com}
\thanks{$^*$Corresponding author: K. Khalil; \texttt{kamal.khalil.00@gmail.com}}
{\renewcommand{\thefootnote}{} \footnote{
$^{1}$LMAH, University of Le Havre Normandie, FR-CNRS-3335, ISCN, Le Havre 76600, France.}}

\begin{abstract}
 In this work, we introduce a compartmental advection-diffusion network model to describe the propagation of stress in a population situated in two interconnected spatial zones during a disaster situation. The model accounts for interactions enabled by intrinsic transitions and imitations within each zone, as well as migrations between these zones. Using semigroup theory and abstract evolution equations, we prove the local existence, uniqueness, and regularity of the solutions. Additionally, we establish the positivity and $L^1$--boundedness of the solutions. Various numerical simulations are provided to illustrate different stress propagation scenarios, including the implementation of a local control strategy designed to minimize stress levels in a population facing a low-risk culture during a dangerous situation.
\end{abstract}
\maketitle
\tableofcontents
\section{Introduction}\label{intro}
Modeling the dynamics of stress propagation in disaster situations is crucial for effective emergency response planning and risk management. Recent works have extensively studied crowds under stress using mathematical modeling and differential equations. These studies range from microscopic modeling using systems of ODEs for low-density crowds to macroscopic modeling using PDEs for high-density crowds. For more details on these models, refer to recent works \cite{Bellomo0,Bellomo2,Gibelli,MauryBook,DiFrancesco,Liu-Zh-Hu,Masmoudi,K_Khalil} and references therein. The macroscopic approach we adopt in this paper considers a crowd as a locally entire quantity, without recognizing individual differences locally, and is therefore more suitable for studying the movement of an extremely large number of pedestrians. In particular, first-order macroscopic models, introduced by Hughes \cite{Hughes2002} (see also \cite{Coscia}), are based on a continuity-type equation and a density-velocity closure equation with suitable boundary conditions. Furthermore, several models are devoted to studying the dynamics of multiple pedestrian species in the context of macroscopic first-order systems, see \cite{Burger,Burger2,Pietschmann,DiFrancesco,Gomes,K_Khalil,Hughes2002,Rosini,Sim_Lan_Hug} and references therein.
 For instance, in \cite{K_Khalil}, the authors introduced a macroscopic first-order model describing stress (panic) propagation (in one spatial zone) by considering interactions between different human behaviors (alert, panic, and control) with advection, diffusion, intrinsic transitions, and imitation parameters and their impact to the propagation of stress in one spatial zone.  These models primarily focus on stress propagation and influence within a single spatial zone during emergency situations. 


In this paper, we present a macroscopic first-order compartmental advection-diffusion-reaction model that describes how populations in two separate interconnected zones $\Omega_1 $ and $\Omega_2 $, via their human behaviors (compartments) that represent the stressed and non-stressed populations respectively, interact via stress contagion within each zone and migration between these two zones, spatial propagation is also considered. More precisely, we consider two bounded zones $\Omega_1 $ and $\Omega_2 $ subsets of $\mathbb{R}^2$ such that $\partial \Omega_1 \cap \partial \Omega_2 =\emptyset$, with two compartments in each zone, one for the stressed population and the other for the non-stressed population.  The mathematical model includes, in each zone, diffusion, which represents the way in which people move randomly in all directions, and non-linear advection, which models the way in which people move closer to departure regions in order to migrate towards arrival regions. It also includes reaction terms (linear and non-linear) that describe intrinsic transitions in human behaviors that reflect natural reductions or increases in stress levels in pedestrians without taking into account interactions between the two human behaviors (stress and non-stress), while imitation refers to the way stress is propagated by eye contact (in the presence of interactions). This modeling approach has been described in \cite{K_Khalil} to describe the spatial propagation of stress by interactions with other human behaviors in one spatial zone, see also references therein for similar reaction-diffusion models. However, in this article, we also consider migration between the two zones, meaning that the two zones are linked by migration, allowing part of the population to move from one zone to the other one in both directions, taking into account the capacity of each zone. Each zone includes one (or more) departure and reception region(s) to facilitate migration between zones. These may be two countries, two cities or two districts linked by road, rail or air. Departure and arrival areas, such as railway stations or airports, this model is illustrated  well in Figure \ref{alea}. This type of interconnected coupling between two zones follows the approaches described in \cite{Irmand,freedman1999global} in the case of reaction-diffusion systems. A key feature of our model is that migration is treated as an internal connection, not by a boundary condition. We also use nonlinear advection along with diffusion to show how movement of pedestrians spreads in a dangerous situation, rather than just using diffusion alone, which makes the model unique and suitable for macroscopic crowd theory. The term "network" is used because the two zones act like nodes in the model.

Using semigroup theory and abstract evolution equations (see \cite{Lunardi,Nagel,Pazy}), we establish the local existence, uniqueness and regularity of our model's solutions. We also prove positivity using results on invariant sets and a subtangential type condition for advection-diffusion-reaction models due to H. Amann \cite{Amann2}. Finally, we provide a $L^1$--boundedness result for solutions showing that the total mass of the population is conserved.

Furthermore, to illustrate our mathematical model, we present several numerical simulations that show how stress propagates in a population with a low risk culture during a dangerous situation. The focus is on one-way movements, where people move from the more dangerous zone 1 to the safer zone 2. This allows us to see how stress moves from one zone to another through migration, and its impact on human behaviors on the environment. We are also looking at ways of controlling stress in the network. Two control scenarios are tested: in the first, controls are placed in the departure region of zone 1, and in the second, controls are applied in the arrival region of zone 2. These controls are designed to reduce and manage stress. The simulations first show how stress propagates in zone 1, then how it affects zone 2 through migration. Next, we show how the first control scenario reduces stress levels in zone 1 and its effect on zone 2. Finally, we show how the second control scenario reduces stress in zone 2. A comparison between the two control scenarios with respect to uncontrolled model is also provided (see Figure \ref{Others behaviors}).

The rest of the paper is organized as follows. Section \ref{sec:2} introduces the extended model incorporating interactions between two separate zones through migration. Before presenting our final model \eqref{Main Model}, we first describe, in \eqref{Model1}, the dynamics of a crowd in each zone separately, accounting for transition and imitation phenomena. We then examine migration (coupling terms) between the two zones, as outlined in model \eqref{Main Model}. Section \ref{sec:3} presents the mathematical analysis, establishing the local existence, uniqueness, and regularity of the solutions and well as the positivity and the $L^1$--boundedness of solutions. Section \ref{sec:4} provides various numerical simulations to illustrate different stress propagation scenarios. Here, we present the two scenarios with a local control strategy in order to minimize the levels of stress density in the network. Section \ref{sec:5} offers concluding remarks and potential directions for future research.

\section{A network PDE model of a population in stress}\label{sec:2}
In this section, we present a compartmental advection-diffusion model to describe stress propagation in two interconnected spatial zones. We begin by describing the model \eqref{Model1}, which describes the evolution of stress in each individual zone, thus capturing the internal dynamics of each zone. 
We then introduce migration terms (coupling) between the two zones, resulting in the main model \eqref{Main Model}, which describes the global dynamics in the network described by $\Omega_1$ and $\Omega_2$.

\subsection{Model type for a population in one zone}
We consider a model representing the evolution of stress within a population. In this model, each individual in the population can be in one of two states: stressed or not stressed. It is assumed that an individual can enter a state of stress through two main mechanisms.\\
The first mechanism is related to emotional contagion. It describes the influence of the group or crowd on an individual. Thus, a person may adopt a state of stress in response to the stress perceived in others around them. This phenomenon reflects the power of social dynamics, where stress spreads from one individual to another through emotional and psychological interactions, creating a chain reaction within the population.\\
The second mechanism is intrinsic to the individual and depends on their personal characteristics. These characteristics may include factors such as culture, physical and mental health, personal history, or stress management abilities. For example, an individual with a genetic predisposition to anxiety or low resilience to stressful events might develop stress regardless of external influence.\\
By combining these two mechanisms with the spatial propagation as diffusion and advection from first-order macroscopic theory, the model aims to capture the complex dynamics of stress evolution within a population, accounting for both social interactions and individual traits with the spatial propagation. This model could therefore provide a better understanding of collective stress waves and individual vulnerability to stress. \\
The following model represents the evolution of a stressed population density $u_P$ and an unstressed one $u_N$ in a crowd in each single spatial bounded zone $ \Omega \subset \mathbb{R}^2$ where $\partial\Omega$ is its associated boundary. 
\begin{equation} \label{Model1}
	\begin{aligned}
		\begin{cases}
			\dfrac{\partial u_P}{\partial t}=d_P\Delta u_P-  \nabla \cdot (\vec{v}_P(u) u_P) + a \   u_N -bu_P+f(u_P,u_N)  u_N   u_P , &\; t>0, \ x \in \Omega,\\[4mm]
			\dfrac{\partial u_N}{\partial t}=d_N\Delta u_N-\nabla \cdot (\vec{v}_N(u) u_N) - a \  u_N +bu_P-f(u_P,u_N)u_N  u_P , &\; t>0, \ x \in \Omega,\\[4mm]
			\partial_{\vec{\eta}} u_P=\partial_{\vec{\eta}} u_N=0, &\; t>0, \, x \in \partial\Omega, \\[4mm] 
			u_P(0,x)=u_{P,0}(x), \  u_N(0,x)=u_{N,0}(x), &\; x \in \Omega,
		\end{cases}
	\end{aligned}
\end{equation}
where $u_P:= u_P (t,x)$ is the instantaneous local density of pedestrians in stress and $u_N:= u_N (t,x)$ is the one of non-stressed pedestrians, so that, $u=u_P+u_N$ is the accumulated local density in $\Omega$. Moreover, the parameters of this model are described as follows:
\begin{itemize}
	\item  \noindent \textbf{Diffusion.} The diffusion terms are given by $d_P\Delta u_P$ for the population in stress and $ d_N\Delta u_N $ is that of non-stress population, where $d_P > 0$ is the diffusion rate for the stressed population and $d_N > 0$ is that for the unstressed population. We assume that stressed pedestrians diffuse faster than unstressed pedestrians, i.e. $d_P > d_N$, this assumption was introduced by our definition of stress since stressed pedestrians can move faster in different senses with respect to non-stressed pedestrians, see \cite{K_Khalil} for a similar assumption.\\
	\item \noindent  \textbf{Advection.} The advection terms are given by $-  \nabla \cdot (\vec{v}_P(u) u_P)$ for stressed population and $ -\nabla \cdot (\vec{v}_N(u) u_N)  $ for unstressed population, where the local speed vectors $ \vec{v}_i:=\vec{v}_i(x,u) $ for $i \in \{ P, N\} $ such that $$ \vec{v}_i(x,u)= v_i(u) \vec{\nu}_i(x),  \quad x \in \Omega, $$  
	where $v_i (u)$ is the maximal density-dependent speed scalar function and $ \vec{\nu}_i(x) $ is the desired direction of movement respectively for pedestrians $i$, to the escape region inside $\Omega$, where $i\in \{ P, N\} $. Different types of speed-density functions are considered in the literature (see \cite{Buchmueller,Coscia,CristianiPiccoliTosin,Seyfried2005}). Here, we  choose a linear dependence:
	$$ 
	v_i (u ) = v_{i,\max}\left( 1- u\right)\quad \text{ for } i=P,N,
	$$
	where $v_{P,\max}$ and $v_{N,\max}$ are two positive constants such that $v_{P,\max} > v_{N,\max}$ which means that the maximal speed of population in stress is greater than that in non-stress. Similar assumptions on the panic and the control maximum speeds are used in \cite{CristianiPiccoliTosin,Liu-Zh-Hu}.
	\item \textbf{The intrinsic transitions.} It is given by the parameters $a, b>0$, which describe behavioral transitions that depend on individual properties (past experience, level of risk culture, etc.). The coefficient $a$ gives the rate of transition from non-stress behavior to stress behavior, while $b$ represents the rate of transition from non-stress behavior to stress behavior.
	
	\item \textbf{The imitation (or contagion) phenomenon.} Individuals have a tendency to imitate the behaviors of those around them, which means that stress and non-stress are both contagious (through eye contact). We are following the principle of dominant behavior here, i.e., the most adopted behavior is the most imitated. Thus, imitation between stress and non-stress behaviors depends on the ratio of the local densities of the two populations around. Hence, the imitation term is given by: 
	\begin{equation}
		f(u_P,u_N) = \alpha_P \xi \left( \frac{u_P}{u_N+\epsilon} \right) - \alpha_N \xi \left( \frac{u_N}{u_P+\epsilon} \right),
	\end{equation}
	where $\alpha_P \geq  0$, is the rate of imitation from stress to non-stress behavior and $\alpha_N \geq  0$, is the rate of imitation from non-stress to stress behavior. The coefficient $\epsilon \ll 1$ is fixed very small only to avoid singularities, while the function $\xi $ is defined as follows:
	\begin{equation}\label{xi0}
		\xi(s) = \frac{s^2}{1+s^2}, \quad s\in \mathbb{R}.
	\end{equation}
	
	\begin{figure}[h]
		\centering
		\begin{tikzpicture}
			\begin{axis}[
				xlabel=$s$,
				ylabel=$\xi(s)$,
				legend pos=north east,
				grid=both,
				axis lines=middle,
				width=0.6\textwidth,
				height=6cm,
				samples=100,
				domain=0:5
				]
				\addplot[blue, thick, domain=0:5] {exp(x^2)/(1+exp(x^2)};
				\legend{$\xi(s)$}
			\end{axis}
		\end{tikzpicture}
		\caption{Representative graph of the function $\xi$.}
		\label{fig: xi-1}
	\end{figure}
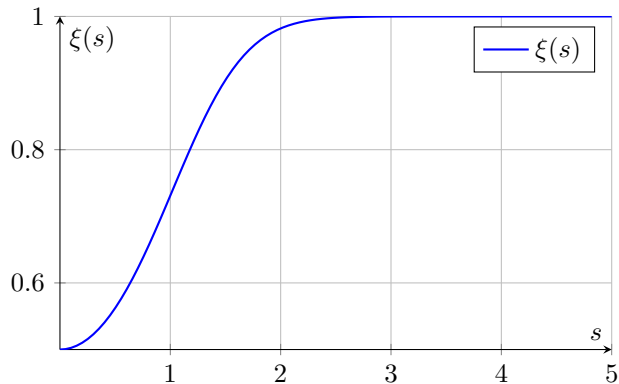
	\item \textbf{Boundary and initial conditions.} The Neumann boundary conditions $$\partial_{\vec{\eta}} u_P=\partial_{\vec{\eta}} u_N =0 $$ indicate that there is no flux of pedestrians across the boundary $\partial \Omega$, where $\partial_{\vec{\eta}} u_i:= \nabla u_i \cdot \vec{\eta}$ for $i\in \{P,N\}$ and $\vec{\eta} $ is the outer normal vector at the boundary $\partial\Omega$. The initial conditions represent the initial population with respect to each human behavior and are given by $u_P(0,\cdot)=u_{P,0}(\cdot)$ and $u_N(0,\cdot)= u_{N,0}(\cdot)$ the accumulated initial population under study. Therefore, by assumptions, we have $$ \int_{\Omega} (u_{P,0}(x) +u_{N,0}(x)) dx = 1 .$$
\end{itemize}
The dynamics in \eqref{Model1} are driven by both diffusive and advective processes, along with the effects of stress transitions  and imitation (contagion). This model can be derived using as in \cite{K_Khalil,Coscia,Hughes2002} and references therein and provides a basis for understanding stress propagation in a single spatial zone. Next, we extend this model to account for interactions between two separate zones through migration.
\subsection{The main network PDE model between two separate zones}
In this section, we examine the evolution of stress in two interconnected spatial zones. The model provides a detailed description of the spread of stress within each zone while accounting for the interactions between the stressed and non-stressed populations, as represented in the model \eqref{Model1}. The  migration (or coupling) between the zones is incorporated through specific terms in the equations, representing the flows of influence between the populations. Let \( \Omega_1 \) and \( \Omega_2 \) be two distinct zones, i.e., $ \partial\Omega_1 \cap \partial\Omega_2=\emptyset$ with associated boundary parts $ \partial \Omega_1 $ and $ \partial \Omega_2 $ respectively. The coupled model for two zones is described by the following system of equations:



\begin{equation}\label{Main Model}
	\begin{aligned}
		\begin{cases}
			\begin{cases}
				&\dfrac{\partial u_{P_1}}{\partial t} = d_{P_1}\Delta u_{P_1} -\nabla \cdot (\vec{v}_{P_1}(u_{P_1})  u_{P_1}) + a_1\  u_{N_1} - b_1 u_{P_1} + f_1(u_{P_1}, u_{N_1})  u_{N_1}  u_{P_1} + f_{21}(u_{P_1}, u_{P_2}), \\[4mm]
				&\dfrac{\partial u_{N_1}}{\partial t} = d_{N_1}\Delta u_{N_1} - \nabla \cdot (\vec{v}_{N_1}  u_{N_1}) - a_1  \ u_{N_1} + b_1 u_{P_1} - f_1(u_{P_1}, u_{N_1}) u_{N_1} u_{P_1} + f_{21}(u_{N_1}, u_{N_2}), \\[4mm]
				&  \partial_{\vec{\eta}} u_{P_1} = \partial_{\vec{\eta}} u_{N_1} = 0, \quad x \in \partial \Omega_1, \\[4mm]
				&u_{P_1}(0) = u_{P_1,0}(x), \quad  u_{N_1}(0) = u_{N_1,0}(x), \quad x \in \Omega_1
			\end{cases} \\[1cm]
			\begin{cases}
				&\dfrac{\partial u_{P_2}}{\partial t} = d_{P_2}\Delta u_{P_2} - \nabla \cdot (\vec{v}_{P_2}  u_{P_2}) + a_2 \ u_{N_2} - b_2 u_{P_2} + f_2(u_{P_2}, u_{N_2})  u_{N_2}  u_{P_2} + f_{12}(x,u_{P_2}, u_{P_1}), \\[4mm]
				&\dfrac{\partial u_{N_2}}{\partial t} = d_{N_2}\Delta u_{N_2} - \nabla \cdot (\vec{v}_{N_2}  u_{N_2}) - a_2 \ u_{N_2} + b_2 u_{P_2} - f_2(u_{P_2}, u_{N_2}) u_{N_2} u_{P_2} + f_{12}(y,u_{N_2}, u_{N_1}), \\[4mm]
				& \partial_{\vec{\eta}} u_{P_2} = \partial_{\vec{\eta}} u_{N_2} = 0, \quad y \in \partial \Omega_2, \\[4mm]
				&u_{P_2}(0) = u_{P_2,0}(y), \quad  u_{N_2}(0) = u_{N_2,0}(y), \quad y \in \Omega_2
			\end{cases}
		\end{cases}
	\end{aligned}
\end{equation}
where the terms are given by:
\begin{itemize}
	\item \( u_{P_1} \) and \( u_{P_2} \): Stressed population densities in zones \( \Omega_1 \) and \( \Omega_2 \), respectively.
	\item \( u_{N_1} \) and \( u_{N_2} \): Non-stressed populations densities in zones \( \Omega_1 \) and \( \Omega_2 \), respectively.
	\item  \( u_1(t) = u_{P_1}(t) + u_{N_1}(t) \) and \( u_2(t) = u_{P_2}(t) + u_{N_2}(t) \):  The accumulated population densities in zones \( \Omega_1 \) and \( \Omega_2 \), respectively.
	\item \( d_{P_1} \) and \( d_{P_2} \): Diffusion coefficients for stressed population in zones \( \Omega_1 \) and \( \Omega_2 \).
	\item \( d_{N_1} \) and \( d_{N_2} \): Diffusion coefficients for the non-stressed population in zones \( \Omega_1 \) and \( \Omega_2 \).
	\item \( \vec{v}_{P_1} \) and \( \vec{v}_{P_2} \): Advection velocities for stress in zones \( \Omega_1 \) and \( \Omega_2 \).
	\item \( \vec{v}_{N_1} \) and \( \vec{v}_{N_2} \): Advection velocities for the non-stressed population in zones \( \Omega_1 \) and \( \Omega_2 \).
	\item \( a_1, b_1\): Intrinsic transitions terms in \( \Omega_1 \).
	\item \( a_2, b_2\): Intrinsic transitions terms in \( \Omega_2 \).
	\item \( f_1, f_2\): Imitation terms in zones \( \Omega_1 \) and \( \Omega_2 \), respectively.
	\item \( f_{12}, f_{21} \): Local migration terms of one zone on the other.
\end{itemize}
The terms diffusion, advection, intrinsic transitions and imitation in each zone $\Omega_i$, $i=1,2$ are described as in the \eqref{Model1} model. To complete this model, we introduce migration (or coupling) terms between the two zones $ \Omega_1$ and $\Omega_2$:
\begin{itemize}
	\item \textbf{Migration (or coupling).}
	The coupling terms represent the migration between zones $\Omega_1$ and $\Omega_2$ in both directions, for stressed and unstressed populations, expressed by \( f_{21} \) and \( f_{12} \). For \( (u_i,u_j)=(u_{N_i},u_{N_j}) \) or \( (u_i,u_j)=(u_{P_i},u_{P_j}) \), \( i,j \in \{1,2 \}, \, i \neq j \), these terms are defined as follows:
	
	\[
	f_{ji}(x,u_i, u_j) = -m_{ji} \mathbf{p}_{ji}(x) u_i  + \varepsilon_i(x) \ m_{ij}\int_{\Omega_i} \mathbf{p}_{ji}(y) u_j(y) \,dy, \; x\in \Omega_i, i,j \in \{1,2 \}, \, i \neq j,
	\]
	where \( f_{ji}(x,u_i, u_j) \) is the local migration term from zone \( \Omega_j \) to \( \Omega_i \) represents the flux balance between the quantity of individuals who have left $\Omega_i$ from location $x$, and the quantity of individuals coming from $\Omega_j$ towards $\Omega_i$ who have arrived at location $x$ within $\Omega_i$. The term $-m_{ji} \mathbf{p}_{ji}(x) u_i$ models the flow of pedestrians leaving $\Omega_i$ from $x$ and the term $ \varepsilon_i(x) \ m_{ij}\int_{\Omega_i} \mathbf{p}_{ji}(y) u_j(y) \,dy $ models the flow of pedestrians entering $\Omega_i$ from $\Omega_j$ and received at location $x$. The parameters are given as follows
	\begin{itemize}
		\item \( m_{ji} \): Maximum proportion of individuals from $\Omega_i$ who can migrate to $\Omega_j$.
		\item \( \mathbf{p}_{ji}(x) \): Probability of leaving \(\Omega_i\) from position \(x\) to \(\Omega_j\). It is assumed to be a continuous function on $\Omega_j$. 
		\item \( \varepsilon_i(x) \): Probability of being received at location \( x \) in $\Omega_i$. It is assumed to be a continuous function on $\Omega_i$.
	\end{itemize}
	
	For example, the migration term for the non-stressed population from $\Omega_1$ to $\Omega_2$ is given by
	\[
	f_{21}(u_{N_1}, u_{N_2}) = -m_{21} u_{N_2} \mathbf{p}_{12}(y) + \varepsilon_2(y) m_{12} \int_{\Omega_1} u_{N_1} \mathbf{p}_{21}(x) \, dx.
	\]
Explicit examples of these functions are given below in Section \ref{sec:4}.
\end{itemize}

\begin{figure}[h]\label{diagram}
	\centering
	\begin{tikzpicture}[scale=0.5]
		\draw[fill=yellow!20] (-7, -5) rectangle (3, 5);
		
		\draw[fill=blue!30] (-2,3) circle (1.5cm) node[above] {\textit{\bf$u_{P_1}$}};
		\draw[fill=red!30] (-2,-3) circle (1.5cm) node[below] {\textit{\bf$u_{N_1}$}};
		
		\draw[->, thick, red] (-3.5,3) to[bend right=60] node[midway, left] {$b_1$} (-3.5,-3);
		
		\draw[<-, thick, red] (-0.5,3) to[bend right=-60] node[midway, right] {$a_1$} (-0.5,-3);
		
		\draw[dashed,<->, thick, red] (-2,1.5) to[bend left=0] node[midway, right] {$f_1$} (-2,-1.5);
		
		\node at (-2,-6) {\bf Zone 1 $(\Omega_1)$};
		
		\draw[fill=green!20] (8, -5) rectangle (18, 5);
		
		\draw[fill=blue!30] (13,3) circle (1.5cm) node[above] { \textit{\bf$u_{P_2}$}};
		\draw[fill=red!30] (13,-3) circle (1.5cm) node[below] {\textit{\bf$u_{N_2}$}};
		
		\draw[<-, thick, red] (14.5,3) to[bend right=-60] node[midway, left] {$a_2$} (14.5,-3);
		
		\draw[->, thick, red] (11.5,3) to[bend right=60] node[midway, right] {$b_2$} (11.5,-3);
		
		\draw[dashed,->, thick, blue] (-2,4.5) to[bend right=-20](13,4.5);
		\draw[dashed,<-, thick, blue] (-2,-4.5) to[bend right=20](13,-4.5);
		\draw[dashed,<-, thick, blue] (-1.7,1.5) to[bend right=20](12.7,1.5);
		\draw[dashed,->, thick, blue] (-1.7,-1.5) to[bend right=-20](12.7,-1.5);
		
		\draw[dashed,<->, thick, red] (13,1.5) to[bend left=0] node[midway, right] {$f_2$} (13,-1.5);
		
		\node at (13,-6) {\bf Zone 2 $(\Omega_2)$};
		\node at (6,6.5) {\bf \color{blue}$f_{21}$ };
		\node at (6,-6.5) {\bf \color{blue} $f_{12}$ };
		\node at (6,1) {\bf \color{blue} $f_{12}$ };
		\node at (6,-1.) {\bf \color{blue}  $f_{21}$ };
	\end{tikzpicture}
	\caption{ Flow diagram for model \eqref{Main Model}. In red, flows generated by territorial interactions: underlining in solid red lines the intrinsic transitions between states of stress and non-stress, and in dashed red lines the flows due to imitation. In blue, migration-related flows.}
	\label{alea}
\end{figure}
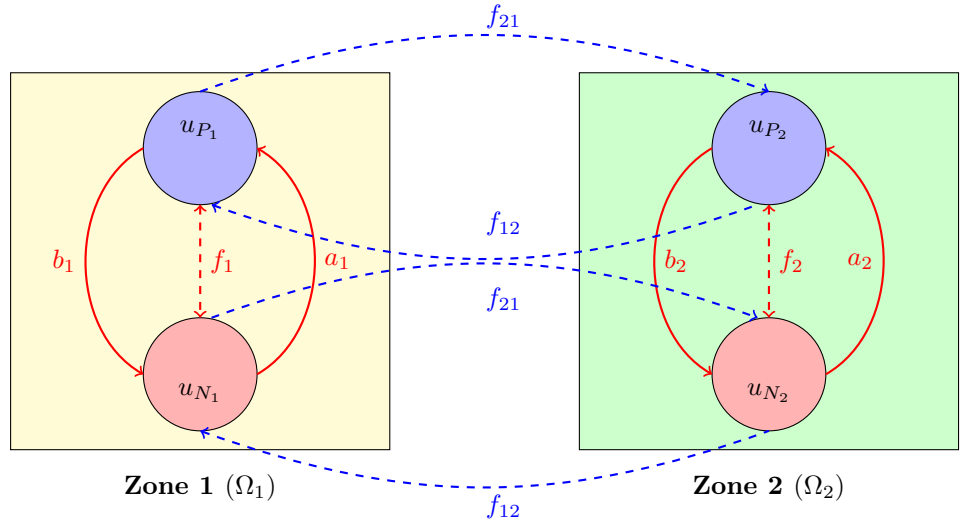
\begin{table}[h!]

	\caption{Functions in the model \eqref{Main Model}}
	
	\label{tab:functions-systeme-ACP}
	
	\centering
	
	\begin{tabular}{l| l }
		
		\textbf{Functions}            &    \textbf{Notation}    \\
		
		\hline
		Imitation function in $\Omega_1 $                 &        $f_{1}$        \\
		
		Imitation function in $\Omega_2 $     &        $f_{2}$    \\
		
		Migration function from $\Omega_1 $ to $\Omega_2$        &    $f_{21}$          \\
		
		Migration function from $\Omega_2 $ to $\Omega_1 $       &    $f_{12}$         \\
		
		Probability of being received at location \(x\) in $\Omega_1$      &    $\varepsilon_{1}(x)$  \\
		
		Probability of being received at location \(y\) in $\Omega_2$      &    $\varepsilon_{2}(y)$    \\
		
		Probability of leaving \(\Omega_1\) from position \(x\) to \(\Omega_2\)  &  \(\mathbf{p}_{21}(x)\) \\
		
		Probability of leaving \(\Omega_2\) from position \(x\) to \(\Omega_1\)  &  \(\mathbf{p}_{12}(y)\) \\ 
		
		The desired direction of pedestrians in $\Omega_1$ from position \(x\) to the exit area & $\vec{\nu}_1(x)$ \\
		
		The desired direction of pedestrians in $\Omega_2$ from position \(y\) to the exit area & $\vec{\nu}_2(y)$ \\

	\end{tabular}
	
\end{table}

%
%
%
%
%
%
%
%
%
%
%
%
%
%
%
%
%
%
%
%




\begin{table}[h!]
	\caption{Parameters of the model \eqref{Main Model}}
	\label{tab:parameters-systeme-ACP}
	\centering
	\begin{tabular}{p{0.35\linewidth}|p{0.25\linewidth}|p{0.25\linewidth}}
		\textbf{Parameters} & \textbf{Notation for zone 1} & \textbf{Notation for zone 2} \\
		\hline
		Diffusion parameters for stressed & $d_{P_1}$ & $d_{P_2}$ \\
		Diffusion parameters for non-stressed & $d_{N_1}$ & $d_{N_2}$ \\
		Maximal speed paramters for stressed & $v_{P_1,max}$ & $v_{P_2,max}$ \\
		Maximal speed paramters for non-stressed & $v_{N_1,max}$ & $v_{N_2,max}$ \\
		Intrinsic evolution from stress to non-stress behaviors & $a_1$ & $a_2$ \\
		Intrinsic evolution from non-stress to stress behaviors & $b_1$ & $b_2$ \\
		Imitation from non-stress to non-stress & $\alpha_{P_1}$ & $\alpha_{P_2}$ \\
		Imitation from stress to non-stress & $\alpha_{N_1}$ & $\alpha_{N_2}$ \\
		Maximum proportion of individuals from $\Omega_i$ who can migrate to $\Omega_j$ & $m_{21}$ & $m_{12}$ \\
	\end{tabular}
\end{table}

\section{Local existence, positivity, and $L^1$--boundedness of the spatio-temporal model \eqref{Main Model}}\label{sec:3} 
In this section, we prove the well-posedness of the spatio-temporal model \eqref{Main Model} introduced previously. Then, we establish the positivity of the solutions and the $L^1$-boundedness of the population densities.

\subsection{The abstract formulation and the associated boundary value Cauchy problem}   
To study the existence and uniqueness of solutions to the system \eqref{Main Model}, we use the abstract formulation and semigroup theory \cite{Nagel,Pazy}. In order to do that, for $ p>2 $, we define the Banach space $ X:= L^p(\Omega_1)^{2} \times L^p(\Omega_2)^{2} $, the product of the Lebesgue spaces of order $p$,  equipped with the following norm 
\[ 
\| \varphi:=(\varphi_1,\varphi_2,\varphi_3,\varphi_4)^{*}\|:=\sum_{i=1}^{4} \| \varphi_i \|_p,
\]
where $\|\cdot \|_p$ is the usual norm in $L^p(\Omega_j)$, $j=1,2 $ and $^*$ designates the vector transpose. 
Moreover, we define the linear closed operator $  (A,D(A )) $ on $X$ by
\begin{equation}
	\left\{
	\begin{aligned}
		A
		&=
		\text{diag}( d_{P_1}  \Delta, d_{N_1}  \Delta, d_{P_2}  \Delta, d_{N_2}  \Delta)\\   
		D(A )& = \lbrace \varphi \in W^{2,p}(\Omega_1)^{2} \times W^{2,p}(\Omega_2)^{2} : \partial_{\vec{\eta_j}} \varphi_i =0\rbrace.
	\end{aligned} \label{Operator (A,D(A))}
	\right. 
\end{equation} 

The nonlinear function $F:[0,\infty)\times X_{\alpha}\longrightarrow X$ is defined by 
\[
F(t,\varphi )(x) = \left\{
\begin{aligned}
	F_{1}(x,\varphi_1(x), \nabla \varphi(x))\; x\in \Omega_1\\
	F_{2}(x,\varphi_2(x), \nabla \varphi(x)), \; x\in \Omega_1\\
	F_{3}(x,\varphi_3(x), \nabla \varphi(x)), \; x\in \Omega_2\\
	F_{4}(x,\varphi_4(x), \nabla \varphi(x)), \; x\in \Omega_2,
\end{aligned}
\right. 
\]
\begin{equation}\label{The function F}
\left\{
\begin{aligned}
	F_{1}(x,\varphi_1(x), \nabla \varphi(x))&= a_1\varphi_2(x) - b_1\varphi_1(x) + [f_1(\varphi_1,\varphi_2)\varphi_1\varphi_2](x) + f_{21}(\cdot,\varphi_1,\varphi_3)(x) \\
	&\quad -  \nabla \cdot (\varphi_1 \vec{v}_{P_1} (\varphi_1+\varphi_2)) (x),\; x\in \Omega_1,\\
	F_{2}(x,\varphi_2(x), \nabla \varphi(x))&= - a_1\varphi_2(x) + b_1\varphi_1(x) - [f_1(\varphi_1,\varphi_2)\varphi_1\varphi_2](x) + f_{21}(\cdot,\varphi_2,\varphi_4)(x) \\
	&\quad -  \nabla \cdot (\varphi_2 \vec{v}_{N_1}(\varphi_1+\varphi_2)) (x),\; x\in \Omega_1, \\
	F_{3}(x,\varphi_3(x), \nabla \varphi(x))&= a_2\varphi_4(x) - b_2\varphi_3(x) + [f_2(\varphi_3,\varphi_4)\varphi_3\varphi_4](x) + f_{12}(\cdot,\varphi_3,\varphi_1) (x) \\
	&\quad - \nabla \cdot (\varphi_3 \vec{v}_{P_2} (\varphi_3+\varphi_4)) (x),\; x\in \Omega_2,\\
	F_{4}(x,\varphi_4(x), \nabla \varphi(x))&= - a_2\varphi_4 (x)+ b_2\varphi_3(x) - [f_2(\varphi_3,\varphi_4)\varphi_3\varphi_4] (x)+ f_{12}(\cdot,\varphi_4,\varphi_2)(x) \\
	&\quad - \nabla \cdot (\varphi_4 \vec{v}_{N_2}(\varphi_3+\varphi_4)) (x),\; x\in \Omega_2.
\end{aligned}
\right.
\end{equation}
and $X_{\alpha}:=\lbrace \varphi \in W^{2\alpha,p}(\Omega_1)^{2} \times W^{2\alpha,p}(\Omega_2)^{2}:  \partial_{n}  \varphi_{i|\partial \Omega} =0  \rbrace$ for some (fixed) $\alpha \in (1/p+1/2,1)$ equipped with the norm 
\[ 
\|\varphi \|_{\alpha}:= \|\varphi\|+\|\nabla \varphi\|+[\varphi]_{\zeta} 
\] 
where 
\[  
[\varphi]_{\zeta}:=\sum_{i=1}^2 \left( \int_{\Omega_1 \times \Omega_1} \dfrac{|\varphi_i(x)-\varphi_i(y) |^p}{|x-y |^{2+p\zeta}} dx dy\right)^{1/p}+\sum_{i=3}^4\left( \int_{\Omega_2 \times \Omega_2} \dfrac{|\varphi_i(x)-\varphi_i(y) |^p}{|x-y |^{2+p\zeta}} dx dy\right)^{1/p},
\] 

$ \zeta=2\alpha -1.$

defines a norm on $X_{\alpha}$ which make it a Banach space. Then by construction, we have the following usefel embedding:
\begin{equation}\label{embedding}
	X_{\alpha}\hookrightarrow  C^1(\overline{\Omega}_1)^{2} \times C^1(\overline{\Omega}_2)^{2}.
\end{equation}

Henceforth, \eqref{Main Model} can be formulated as the following abstract Cauchy problem in $X$:
\begin{equation}\label{Abstract model}
	\left\{
	\begin{aligned}
		\dfrac{d U(t)}{dt} &=A U(t) + F(t,U(t) ), \quad t>0, \,  \\
		U(0) & = U_0 \in X, \\
	\end{aligned} 
	\right. 
\end{equation} 
where $U(t)=(u_{N_1}(t),u_{P_1}(t),u_{N_2}(t),u_{P_2}(t))$ and  the initial condition $U_0=(u_{N_1,0},u_{P_1,0},u_{N_2,0},u_{P_2,0})$ represents the initial population densities of the populations $P_1$, $N_1$, $P_2$, and $N_2$ respectively.
\subsection{Prelimanry results} In this section, we introduce some preliminary results necessary for the main mathematical results of the sequel.

\begin{definition} \label{Positive cones}
	We define the positive cone $X^+$ of $X$ by:
	$$ X^+:=\lbrace \varphi \in L^{p}(\Omega_1)^{2} \times L^{p}(\Omega_2)^{2} :  \varphi_i\geq 0 \; a.e., \; i=1,\cdots, 4 \rbrace.$$
	Hence, the positive cone of $X_{\alpha}$ is given by $$ X_{\alpha}^+:=X^{+}\cap X_{\alpha} .$$
\end{definition}

A semigroup $(e^{t A}(t))_{t\geq 0}$, is said to be positive in $X$, if and only if, for every $ \varphi  \in X^+ $,  implies $ e^{t A}(t) \varphi \in X^+ $ for all $ t\geq 0$.\\

Now, using the generation result of the $L^p$--realization of the Laplacian operator due to Davies \cite{Davies}, we give the generation result of the operator $(A,D(A))$ on $X$. 
\begin{proposition}\label{Proposition semigroup}
The operator $(A,D(A)) $ generates a positive holomorphic $C_0$-semigroup $(e^{tA})_{t\geq 0} \subset \mathcal{L}(X) $ on $X$.
\end{proposition}
\begin{proof}
The proof is obvious, since the operator $(A,D(A))$ on $X$ is a diagonal type operator given by \eqref{Operator (A,D(A))}. So, the fact that the $L^p$--realization of Laplacian operator in a bounded smooth domain $\Omega \subset \mathbb{R}^2$ with homegeneous Neumann boundary conditions generates a positive holomorphic $C_0$-semigroup $(e^{t \Delta})_{t\geq 0}$, see \cite{Davies}.\\
Therefore, the semigroup $(e^{t A})_{t\geq 0}$ is given by the following matrix-valued operators 
\begin{equation*}
e^{t A}(t)=diag(e^{t A_1}(t),\cdots ,e^{t A_4}(t) )^*, \quad t\geq 0,
\end{equation*}
where, for each $i=P_1,N_1,P_2,N_2  $, $ (e^{t A_i}(t) )_{t\geq 0}$ is the semigroup generated by the operator $  A_i=d_i \Delta $ in $L^p(\Omega_1)$ for $i=P_1,N_1  $  and $  A_i=d_i \Delta $ in $L^p(\Omega_2)$ for $i=P_2,N_2  $. 
\end{proof}
The following result shows the local Lipschiz continuity charatcter of the nonlinear term $F$ given by \eqref{The function F}.
\begin{lemma}\label{Lemma Local Lipschitz}
	The function $F$ is Lipschitzian in bounded sets i.e., for all $R>0$ there exists $L_R \geq 0$ such that 
	$$ \| F(\varphi)-F(\psi) \| \leq L_R \|\varphi-\psi \|_{\alpha},  \quad \varphi, \psi\in B(0,R), \, t\geq 0. $$
\end{lemma}

\begin{proof}
Let \( F(t, \varphi) \) be the nonlinear function as defined in \eqref{The function F}, and let \( B(0,R) \subset X^\alpha \) denote the centred ball of radius \( R >0 \) in the norm \( \|\cdot\|_\alpha \). We aim to show that there exists a constant \( L_R \geq 0 \) such that
\[
\| F(\varphi) - F(\psi) \| \leq L_R \|\varphi - \psi\|_\alpha, \quad \varphi, \psi \in B(0, R), \quad t \geq 0.
\]
Indeed, the function \( F(t, \varphi) \) is composed of several nonlinear terms: advection, intrinsic transition, imitation, and migration. We estimate the difference \( \| F(\varphi) - F(\psi) \| \) term by term. Note that the following constants \( C_{i,R} \), for $ i=1,\cdots,1 $, depend on the radius \( R \), since \( \varphi, \psi \in B(0,R) \) and on $\Omega_1$ and $\Omega_2$.\\
The advection terms involve \( \nabla \cdot (\varphi_i \vec{v}_{P_i}) \) and \( \nabla \cdot (\varphi_i \vec{v}_{N_i}) \) where \( \vec{v}_{P_i} \) and \( \vec{v}_{N_i} \), for $i=1,2 $, depend on \( \varphi_1 + \varphi_2 \) or \( \varphi_3 + \varphi_4 \) respectively. So by definition, we have:
  \[
  \| \nabla \cdot (\varphi_1 \vec{v}_{P_1}) - \nabla \cdot (\psi_1 \vec{v}_{P_1}) \| \leq C_{1,R} \left(  \|\varphi - \psi\|_\infty +\|\nabla\varphi - \nabla\psi\|_\infty\right).
  \]
Similarly for the other advection terms.\\
 The intrinsic transition terms are linear and take the form \( b_i \varphi_i \) and \( a_i \varphi_i \). These are trivially Lipschitz continuous:
  \[
  \| b_i \varphi_1 - b_i \psi_1 \|_p = b_i \| \varphi_1 - \psi_1 \|_p \leq C_{2} \|\varphi - \psi\|_\infty,
  \]
\[
  \| a_i \varphi_2 - a_i \psi_2 \|_p = a_i \| \varphi_1 - \psi_1 \|_p \leq C_{3} \|\varphi - \psi\|_\infty,
  \] 
The imitation terms involve expressions like \( f_1(\varphi_1, \varphi_2) \varphi_1 \varphi_2 \) and \( f_2(\varphi_3, \varphi_4) \varphi_3 \varphi_4 \). Since \( f_1 \) is smooth and Lipschitz continuous in bounded sets, we have:
  \[
  \| f_1(\varphi_1, \varphi_2) \varphi_1 \varphi_2 - f_1(\psi_1, \psi_2) \psi_1 \psi_2 \|_p \leq C_{4,R} \|\varphi - \psi \|_\infty.
  \]
Similarly, for the term \( f_2(\varphi_3, \varphi_4) \varphi_3 \varphi_4 \).\\
Migration between zones involves terms like \( f_{21}(\varphi_1, \varphi_3) \), \( f_{21}(\varphi_2, \varphi_4) \), \( f_{12}(\varphi_1, \varphi_3 \) and \( f_{12}(\varphi_2, \varphi_4) \). Since, for example, \( f_{21} \) is Lipschitz continuous in bounded sets, we have:
  \[
  \| f_{21}(\cdot,\varphi_1, \varphi_3) - f_{21}(\cdot,\psi_1, \psi_3) \|_p \leq C_{5,R} \|\varphi - \psi \|_\alpha.
  \]
Similarly, for the other migration terms.\\  
By Sobolev embedding \eqref{embedding}, \( X^\alpha \) is continuously embedded into \(C^1(\overline{\Omega}_1)^{2} \times C^1(\overline{\Omega}_2)^{2}\) equipped with its usual norm. This guarantees that the products of \( \varphi_i \), their gradients, and the nonlinear terms are Lipschitz continuous in \( B(0, R) \).

Summing up all the contributions from the advection, diffusion, imitation, and migration terms, we conclude that:
\[
\| F(\varphi) - F(\psi) \| \leq L_R \|\varphi - \psi\|_\alpha,
\]
where \( L_R \geq 0\) depends on the Lipschtiz constants for the terms before and on the embedding constant. This shows that \( F \) is locally Lipschitz continuous in \( B(0, R) \subset X^\alpha \).\\
This completes the proof.
\end{proof}

\subsection{Main results} In this section we give some prelilmanry results about thexistence, uniqueness, regularity and positivity of the soulutions as well as $L^1$-boundedness.
The following result shows local existence and regularity results for our abstract model \eqref{Abstract model}.
\begin{theorem}[Local existence and regularity]\label{Theorem exist regu}
	For each $ U_{0} \in X_\alpha $ there exist $ T_{U_{0}} >0 $ and a unique maximal (classical) solution $U(\cdot):=U(\cdot ,U_{0}) \in C([0,T_{U_0}),X_{\alpha})\cap C^{1}((0,T_{U_0}),X)$ of equation \eqref{Main Model} such that
	\begin{equation} \label{integral formulation of solution}
		U(t)=e^{t A}U_{0}+ \int_{0}^{t} e^{(t-s) A}F(U(s)) ds, \quad t\in [0,T_{U_0}).
	\end{equation} 
	Moreover the solution $U$ satisfies the following property:
	\begin{equation}
		T_{U_0}=+\infty \quad \text{or} \quad  \limsup_{t\rightarrow T_{U_0}^{-} } \| U(t )\| = +\infty . \label{ETFFormula}
	\end{equation} 
\end{theorem}
\begin{proof}
The result holds directly from \cite[Theorem 2.1.1]{T.Dlotko} and also \cite[Theorem 7.1.2]{Lunardi}.
\end{proof}
%
Now, we present a positivy result of the local classical solutions of the model \eqref{Abstract model}, guaranted by Theorem \ref{Theorem exist regu}, provided that the initial condition is positive in $X_\alpha$. We notice that, our result uses the subtangential condition and invariant closed sets for advection-diffusion-reaction models due to H. Amann \cite{Amann2}. We recall that, Amann's result generalize that of R. H. Martin \cite{Martin} in the case of reaction-diffusion equations (i.e., without nonlinear advection).
\begin{theorem}[Positivity]\label{Positivity}
For each $ U_0 \in X_{\alpha}^{+} $ equation  \eqref{Main Model}  has a unique maximal solution $U(\cdot ,u_{0}) \in C([0,T(U_0 )),X_{\alpha})\cap C^{1}((0,T(U_0 )),X)$  such that $ U(t)\in  X_{\alpha}^{+} $ for all $ t\in [0,T(U_0 ) ) $.
\end{theorem}
\begin{proof}
	To prove this result, we use the invariant closed set result for semilinear advection-diffusion-reaction equations due to H. Amann, see \cite{Amann3}. So, it  suffices to prove that the following invariant conditions hold:
	\begin{itemize}
		\item[(i)] $ e^{tA} X_\alpha^+ \subset X_\alpha^+  $ for all $t\geq 0$.
		\item[(ii)] For all $\varphi \in X^+_\alpha $, $\lim_{h\rightarrow 0} h^{-1}d(\varphi + h F(\varphi);X^{+})=0$.
	\end{itemize}
	The statement (i) follows directly from Proposition $\ref{Proposition semigroup}$ and the fact that $ e^{tA} X_\alpha \subset X_\alpha $ for all $t\geq 0$. To show that the subtangential condition (ii) holds, we argue as follows:\\
	Let us consider the (pointwize) subtangential conditions given by:\\
	For all $\varphi:=(\varphi_1,\varphi_2,\varphi_3,\varphi_4)^{*} \in X^+_\alpha $, 
	\begin{equation} \label{subtang cond 11}
		\lim_{h\rightarrow 0} h^{-1}d(\varphi(x) + h F_i(x,\varphi_i (x),\nabla \varphi (x));[0,+\infty))=0, \quad \text{for all } x \in \Omega_1, \, i=1,2,
	\end{equation}
	and 
	\begin{equation}\label{subtang cond 12}
		\lim_{h\rightarrow 0} h^{-1}d(\varphi(x) + h F_i(x,\varphi_i (x),\nabla \varphi(x));[0,+\infty))=0, \quad \text{for all } x \in \Omega_2, \, i=3,4.
	\end{equation}
	Hence, the subtangential conditions \eqref{subtang cond 11} and \eqref{subtang cond 12} hold, since for all $x\in \Omega_1$ and $z \in \mathbb{R}$,  $ F_i(x,0,z)=0 $ for $  i=1,2 $ and,  for all $x\in \Omega_2$ and $z \in \mathbb{R}$, $ F_i(x,0,z)=0 $ for  $  i=3,4$, we refer to \cite[Section 1]{Amann3} for more details. \\
	Therefore, the subtangential condition (ii) holds using the same proof as in \cite[Lemma 7.4, Section 7]{Amann3}. Hence, the result yields from \cite[Theorem 1]{Amann3}, see also \cite[Section 3.3]{K_Khalil} for another similar proof.
\end{proof}

Next, we give the $L^1$-boundedness of the solutions guarantees the conservation of mass of in our model.

\begin{proposition}[$L^1$-boundedness] Let the map $V:[0,T(U_0 ))\longrightarrow [0,+\infty)$ defined by 
	
	$$ V(t)=\int_{\Omega_1} \left[ u_{P_1}(t,x)+u_{N_1}(t,x) \right] dx + \int_{\Omega_2} \left[ u_{P_2}(t,y)+u_{N_2}(t,y) \right] dy.$$
	Then, we have
	\[
	V(t) =1 \quad \text{ for all } t\in  [0,T(u_0 ) ).
	\]
\end{proposition}

\begin{proof}
	Let \( U_{0} \in X_{\alpha}^{+} \) and let \(U(t,U_0)=(u_{P_1}(t,\cdot), u_{N_1}(t,\cdot), u_{P_2}(t,\cdot), u_{N_2}(t,\cdot))^{*}\) for all \( t\in [0,T(U_0)) \) be the corresponding maximal solution. It is clear, from Theorem \ref{Theorem exist regu}, for all \( t\in [0,T(U_0)) \), that \(U(t,U_0) \in X_{\alpha}\hookrightarrow C^{1}(\overline{\Omega_1})^2\times C^{1}(\overline{\Omega_2})^2\). This means that, for all $t\in t\in [0,T(U_0))$, \(u_i(t,\cdot)\) for \( i=P_1,N_1 \), is bounded with respect to \(\Omega_1\) and, for all $t\in [0,T(U_0))$, $u_i(t,\cdot )$ , for \( i=P_2,N_2 \), is bounded with respect to \(\Omega_2\) and then are bounded in \(L^1(\Omega_1)\) and \(L^1(\Omega_2)\) norm respectively. Hence, we consider the following mappings:
	
	\[
	t\in [0,T(u_0)) \longmapsto V_1(t):= \int_{\Omega_1} \left[ u_{P_1}(t,x)+u_{N_1}(t,x) \right] dx \in \mathbb{R},
	\]
	
	\[
	t\in [0,T(u_0)) \longmapsto V_2(t):= \int_{\Omega_2} \left[ u_{P_2}(t,y)+u_{N_2}(t,y) \right] dy \in \mathbb{R},
	\]
	and
	\[
	t\in [0,T(u_0)) \longmapsto V(t):= V_1(t) + V_2(t) \in \mathbb{R},
	\]
	By assumption on the initial conditions, we have 
	\[
	V(0)=\int_{\Omega_1} u_{P_1} (x) dx + \int_{\Omega_2} u_{N_2} (y) dy = 1.
	\]
	
	Additionally, the positivity result in Proposition \ref{Positivity} gives 
	\[
	V(t) \geq 0  \quad \text{ for all } t\in [0,T(U_0 ) ).
	\]
	
	Thus, the mapping $V$ is well-defined. Moreover, the mappings \( V_1 \) and \( V_2 \) are well-defined and continuously differentiable on \([0,T(U_0)) \). Hence, we show that 
	\[
	\dfrac{d}{dt}V(t) = \dfrac{d}{dt}V_1(t) + \dfrac{d}{dt}V_2(t) = 0, \quad t \in [0,T(U_0 ) ).
	\]
	
	Indeed, using the Green-Ostrogradski formula, we obtain that 
	\begin{align*}
		\dfrac{d}{dt}V_1(t) &= \int_{\Omega_1} \partial_t \left[ u_{P_1} + u_{N_1} \right] dx \\
		&= d_{P_1}\int_{\Omega_1} \Delta u_{P_1} dx - \int_{\Omega_1} \nabla (u_{P_1}\vec{v}_{P_{1}}(u_1)) dx + d_{N_1}\int_{\Omega_1} \Delta u_{N_1} dx - \int_{\Omega_1} \nabla (u_{N_1}\vec{v}_{N_{1}}(u_1)) dx \\
		&\quad + \int_{\Omega_1} f_{21}(u_{P_1}, u_{P_2}) dx + \int_{\Omega_1} f_{21}(u_{N_1}, u_{N_2}) dx \\
		&= \int_{\Omega_1} f_{21}(u_{P_1}, u_{P_2}) dx + \int_{\Omega_1} f_{21}(u_{N_1}, u_{N_2}) dx.
	\end{align*}
	
	Similarly, we have 
	\begin{align*}
		\dfrac{d}{dt}V_2(t) &= \int_{\Omega_2} \partial_t \left[ u_{P_2} + u_{N_2} \right] dx \\
		&= d_{P_2}\int_{\Omega_2} \Delta u_{P_2} dx - \int_{\Omega_2} \nabla (u_{P_2}\vec{v}_{P_{2}}(u_2)) dx + d_{N_2}\int_{\Omega_2} \Delta u_{N_2} dx - \int_{\Omega_2} \nabla (u_{N_2}\vec{v}_{N_{2}}(u_2)) dx \\
		&\quad + \int_{\Omega_2} f_{21}(u_{P_2}, u_{P_1}) dx + \int_{\Omega_2} f_{21}(u_{N_2}, u_{N_1}) dx \\
		&= \int_{\Omega_2} f_{21}(u_{P_2}, u_{P_1}) dx + \int_{\Omega_2} f_{21}(u_{N_2}, u_{N_1}) dx.
	\end{align*}
	
	Thus, by definition of $f_{12}$ and $f_{21}$, we obtain that,
	\[
	\dfrac{d}{dt}V(t) = \dfrac{d}{dt}V_1(t) + \dfrac{d}{dt}V_2(t) = 0.
	\]
	
	Therefore
	
	\[
	V(t) = V(0)=1 \quad \text{ for all } t\in  [0,T(u_0 ) ).
	\]
	
\end{proof}
\section{Numerical simulations}\label{sec:4}
In this section, we aim to illustrate our model \eqref{Main Model} numerically. Specifically, in our simulations, we focus on the unidirectional migration of stress from zone \(\Omega_1\) to zone \(\Omega_2\). Populations leave \(\Omega_1\) from a departure area, modeled by the function \(\mathbf{p}_{21}\), and are received in a reception area within \(\Omega_2\), modeled by the function \(\varepsilon_2\). The local propagation in each zone considers transition, imitation, diffusion, and advection in \(\Omega_1\).

First, we visualize how stress propagates from \(\Omega_1\) to \(\Omega_2\) and its impact on the stress level of the population in \(\Omega_2\). Second, we provide a local control strategy aimed at reducing the high stress level in \(\Omega_2\) caused by this migration. We consider two scenarios: initially, control measures are applied locally at the departure area in \(\Omega_1\) to measure the impact of this strategy on stress levels in \(\Omega_2\). Subsequently, a local control is implemented at the arrival area in \(\Omega_2\) to reduce stress levels in \(\Omega_2\) caused by migration.

Note that our numerical simulations use parameters that reflect a low-risk culture of the population, where pedestrians tend to be more stressed. Please refer to the parameters outlined in Table \ref{tab:parameters-systeme-ACP} for details on the specific values used in our simulations. The numerical simulations were carried out using FreeFem++, and the visualizations were done using ParaView and R.

\subsection{Evolution of stress by unidirection migration from $\Omega_1$ to $\Omega_2$ (Wc)}\label{sec:4.1}
Imagine a scenario where a disaster has occurred in \(\Omega_1\), causing a high level of stress. This stress spreads through transition, imitation, and diffusion phenomena as pedestrians in \(\Omega_1\) progressively try to escape to the departure area \(\mathbf{p}_{21}\) by advection. Consequently, \(\Omega_2\) becomes the reception domain due to immigration phenomena, where the initial stress level is lower compared to \(\Omega_1\). As people move into \(\Omega_2\), stress begins to propagate in this zone through similar transition, imitation, and diffusion processes for stressed people coming from \(\Omega_1\) to the arrival zone \(\varepsilon_2\), highlighting the spread of stress across interconnected zones during a disaster. In the following, we present the explicit parameters used in our numerical simulations:\\
\hfill \textbf{The departure and the arrival areas.} The functions \(\mathbf{p}_{21}\) and \(\varepsilon_2\) share the same definition, (are of Guassian typa functions) with the only difference being that they are defined over different domains. Each models a circular area centered at a given point and bounded by a radius.

We define the function \(\mathbf{p}_{21}\) on the domain \(\Omega_1\) with values in the interval \([0, 1]\) as follows:
\[
\mathbf{p}_{21}(x) = \exp\left(-\frac{|x - a_1|^2}{r_1^2}\right), \quad x \in \Omega_1,
\]
where \(a_1 \in \Omega_1\) is the center of the departure area and \(r_1\) is the radius defining this area around \(a_1\). \(|x - a_1|\) represents the Euclidean norm of \(x - a_1\). In simple terms, the departure area is a circle centered at \(a_1\) with radius \(r_1\). The function \(\mathbf{p}_{21}\) models the probability of departures in this area: the closer to the center, the higher the probability; the farther from the center, the lower the probability.

The function \(\varepsilon_2\) on the domain \(\Omega_2\) with values in the interval \([0, 1]\) is defined as:
\[
\varepsilon_2(y) = \exp\left(-\frac{|y - a_2|^2}{r_2^2}\right), \quad y \in \Omega_2,
\]
where \(a_2 \in \Omega_2\) is the center of the reception area on \(\Omega_2\) and \(r_2\) is the radius defining this area around \(a_2\). However, we can choose any other function, provided it is continuous and has the same properties as a probability. In Figure \ref{fig:zone-dep-arr}, we can observe an example of $1$--D  (Figure \ref{fig:zone-dep-arr}-(A)) and $2$--D  (Figure \ref{fig:zone-dep-arr}-(B)) for an example of functions \(\mathbf{p}_{21}\) and \(\varepsilon_2\) that are of Gaussian type.
\begin{figure}[htbp]
	\centering
	\begin{subfigure}[b]{0.48\textwidth}
		\centering
		\begin{tikzpicture}
			\begin{axis}[
				width=0.75\textwidth,
				height=0.75\textwidth,
				xlabel={$x$},
				ylabel={$f(x)$},
				domain=-10:10,
				samples=100,
				xmin=-10, xmax=10,
				ymin=0, ymax=1.4,
				ytick={0,0.2,0.4,0.6,0.8,1,1.2},
				]
				\addplot[
				thick,
				smooth,
				color=red,
				]
				{exp(-((x - 0.5)^2)/4^2)};
			\end{axis}
		\end{tikzpicture}
		\caption{Representative curve of the function 
		$f(x)=\exp\bigg(-\dfrac{|x - 0.5|^2}{4^2}\bigg)$}
	\end{subfigure}
	\hfill 
	\begin{subfigure}[b]{0.48\textwidth}
		\centering
		\begin{tikzpicture}
			\begin{axis}[
				width=0.75\textwidth,
				height=0.75\textwidth,
				view={0}{90},
				xlabel={$x_1$},
				ylabel={$x_2$},
				domain=-10:10,
				samples=50,
				colormap/hot,
				colorbar,
				xmin=-10, xmax=10,
				ymin=-10, ymax=10,
				]
				\addplot3[
				surf,
				shader=interp,
				]
				{exp(-((x - 0.5)^2 + (y - 0.5)^2)/4^2)};
			\end{axis}
		\end{tikzpicture}
		\caption{Representation of the function $f(x_{1},x_{2})=\exp\left(-\frac{(x_{1} - 0.5)^2 + (x_{2} - 0.5)^2}{4^2}\right).$ }
	\end{subfigure}
	\caption{Example of a function modeling departure and reception probabilities. The redder the zone, the higher the probability of leaving that area, and the bluer the zone, the lower the probability of leaving that location.}
	\label{fig:zone-dep-arr}
\end{figure}
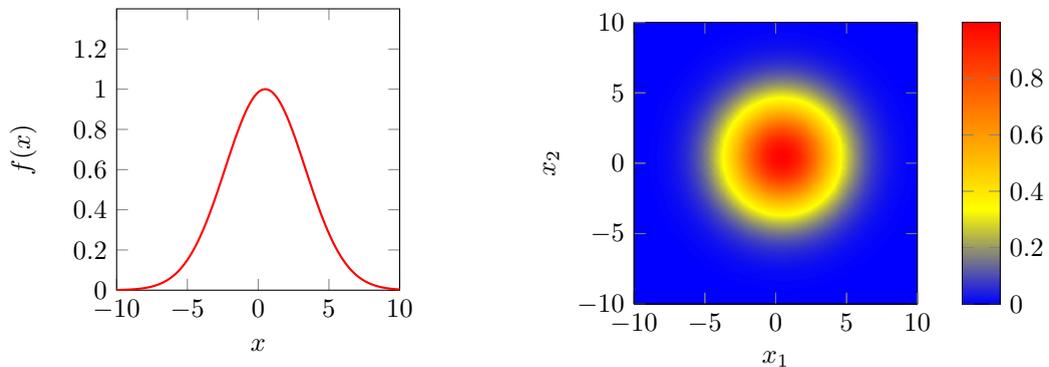

\noindent \textbf{The desired escape direction of the movement.} The desired direction of the movement of the population at each position \((x_1, x_2)^*\) in zone \(\Omega_1\) toward the departure area, a disc centered at \((a, b)^*\), is given by the vectors \(\vec{\nu}(x_1, x_2)\). This can be expressed as follows:
\[
\vec{\nu}(x_1, x_2) = 
\begin{cases} 
	\left( \frac{a - x_1}{\sqrt{(a - x_1)^2 + (b - x_2)^2}}, \frac{b - x_2}{\sqrt{(a - x_1)^2 + (b - x_2)^2}} \right) & \text{if} \ (x_1, x_2) \in \Omega_1, \\
	0 & \text{if} \ (x_1, x_2) \in \partial \Omega_1.
\end{cases}
\]
In \(\Omega_2\), the population only diffuses randomly in all directions from the arrival area. There is no advection, since the migration is only unidirectional. See Figure \ref{fig:zones} for more details.
\begin{figure}[h!]
	\centering
	\includegraphics[width=0.75\linewidth]{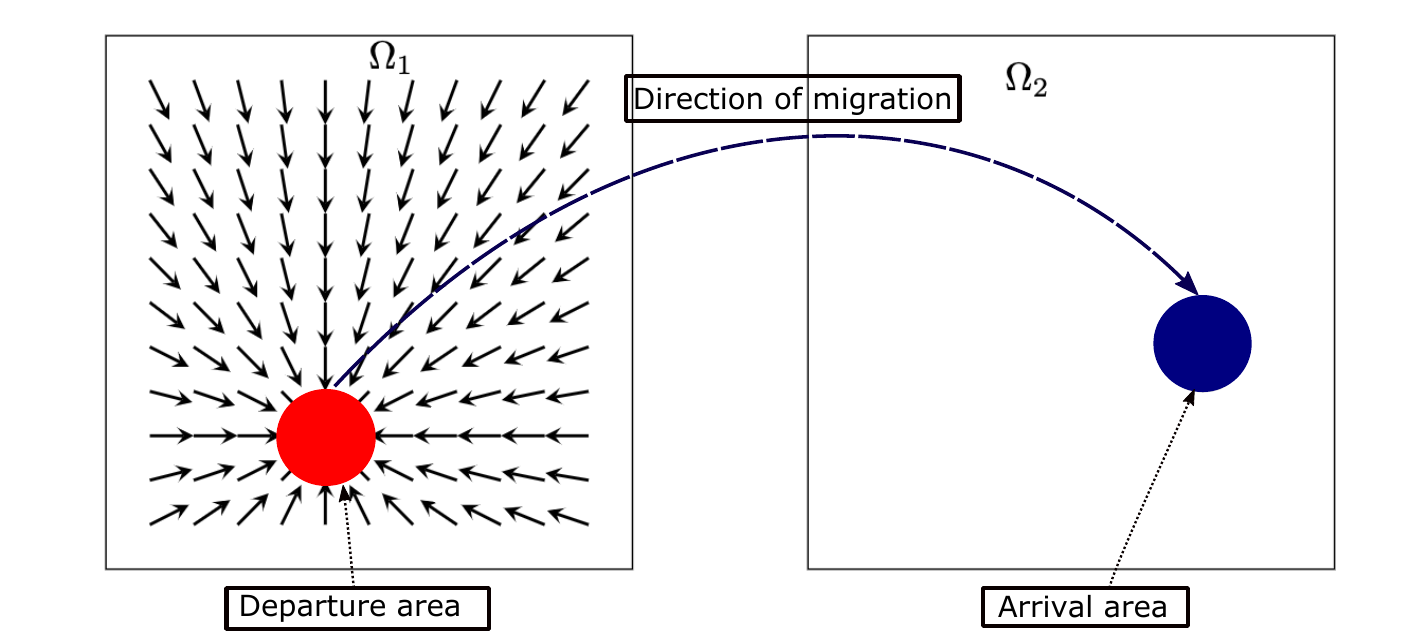}
	\caption{\textbf{Network map.} the square on the left models \(\Omega_1\) and the one on the right models \(\Omega_2\). The red circle in \(\Omega_1\) represents the departure area modeled by the function \(\mathbf{p}_{21}\), and the vector field indicates the desired direction of the movement \(\vec{\nu}\) of the population in \(\Omega_1\) towards the escape area (departure area) centered at \((a, b)\). The blue circle in \(\Omega_2\) represents the reception area for migrants from domain \(\Omega_1\). The red arrow indicates the direction of migration from \(\Omega_1\) to \(\Omega_2\).}
	\label{fig:zones}
\end{figure}

\noindent \textbf{Initial conditions.} The initial conditions are given as follows:\\
In \(\Omega_1\), we assume that at \(t=0\), the initial population is entirely stressed, with the non-stressed population being zero. Furthermore, the initial distribution of the stressed population is divided into three separate clusters.\\
In \(\Omega_2\), we assume that at \(t=0\), the initial population is entirely non-stressed, with the stressed population being zero. Additionally, the initial distribution of the non-stressed population is also divided into three separate clusters.
For more details, we refer to Figure \ref{Initial_Conditions}.
\begin{figure}[htbp]
\begin{center}
\begin{minipage}[c]{.4\linewidth}
\begin{center}
"Initial stressed population"\\
 \hspace{-1cm }$\Omega_1 $  \quad \quad \quad \quad $\Omega_2 $
\end{center}
\includegraphics[width=4cm,height=2cm]{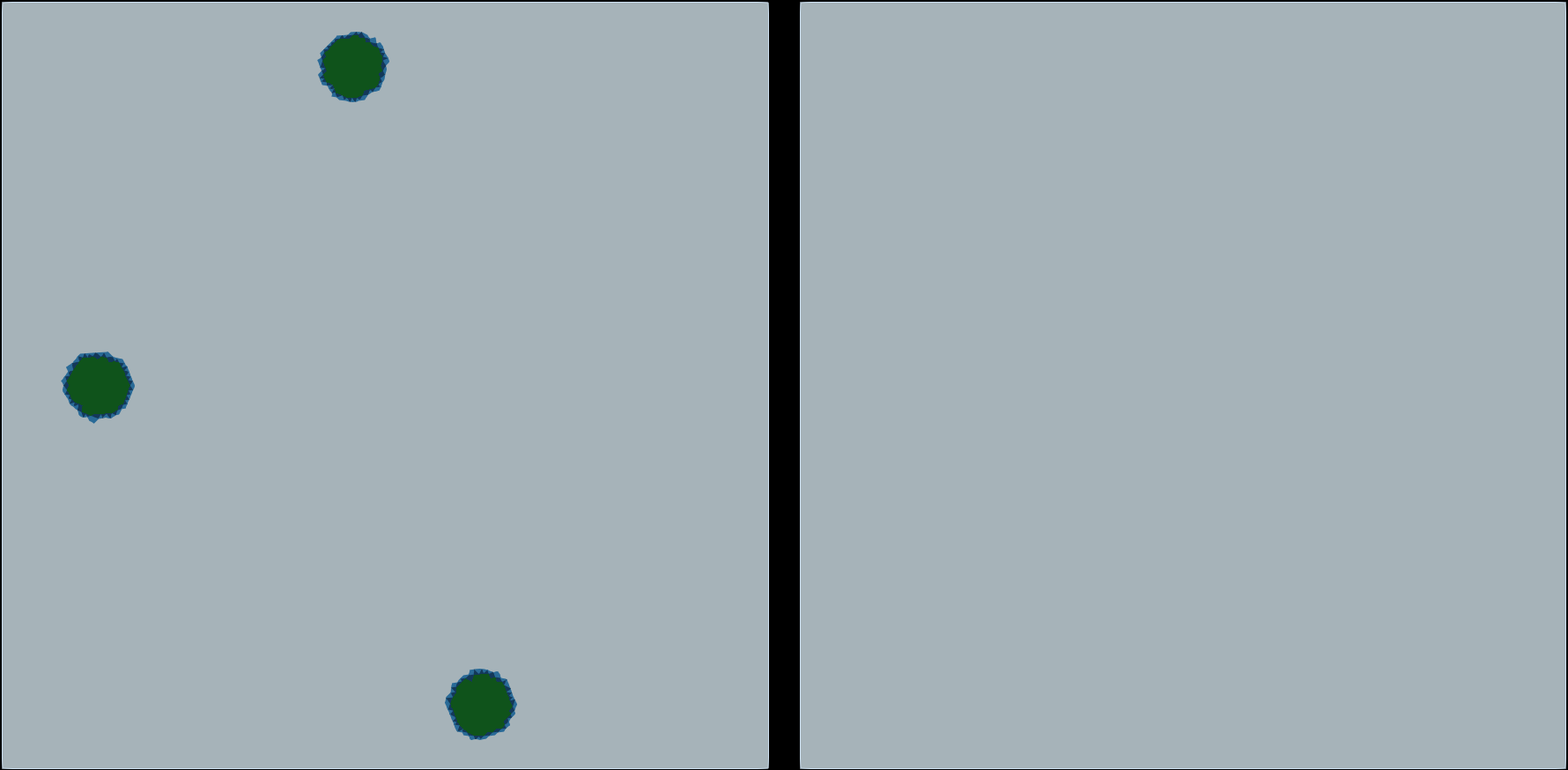}
\end{minipage}
\begin{minipage}[c]{.4\linewidth}
\begin{center}
"Initial non-stressed population"\\
 \hspace{-1cm }$\Omega_1 $  \quad \quad \quad \quad $\Omega_2 $
\end{center}

\includegraphics[width=4cm,height=2cm]{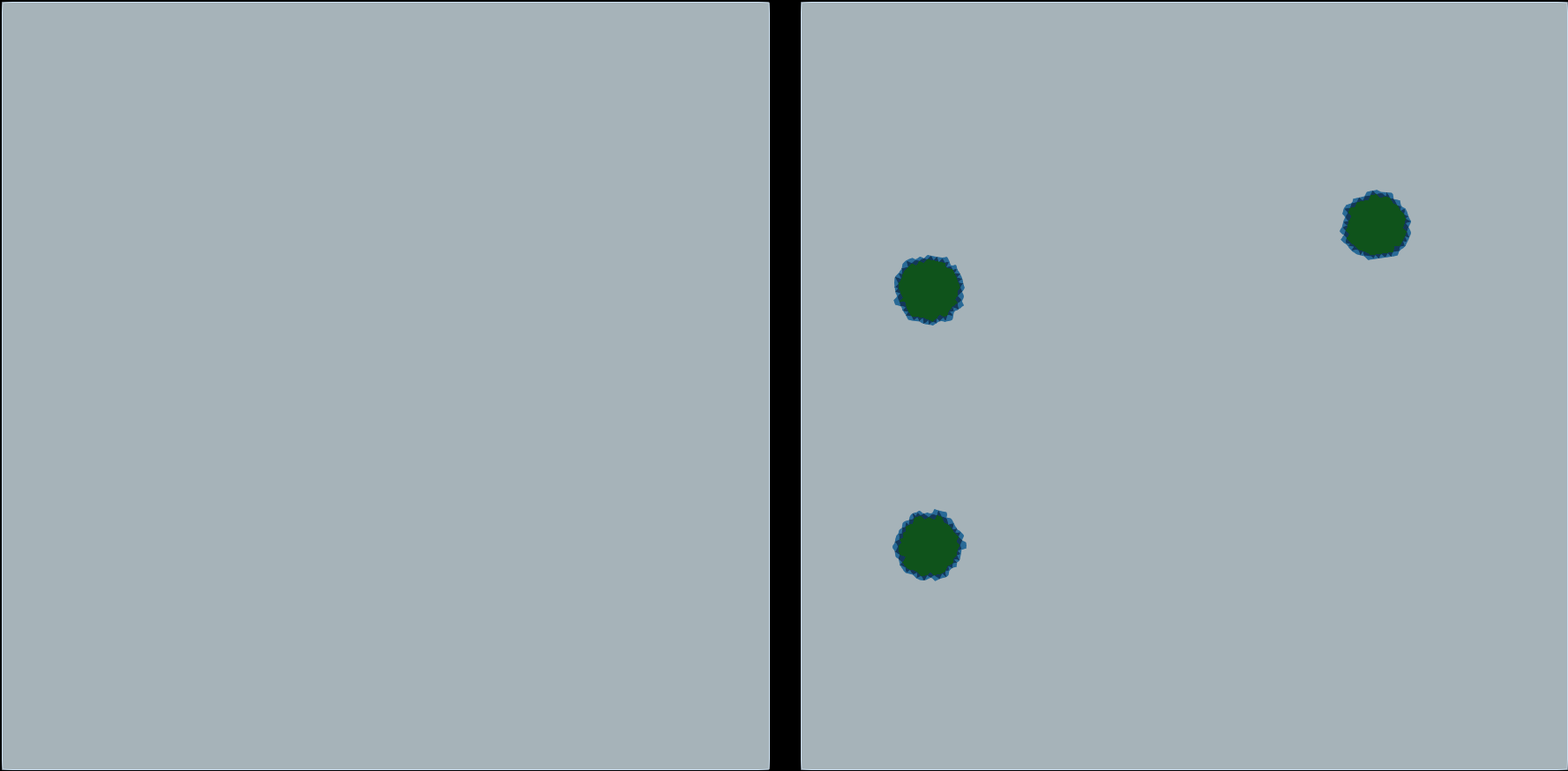}
\end{minipage}

\begin{center}
\begin{minipage}[c]{.5\linewidth}
\includegraphics[width=5cm,height=1.5cm]{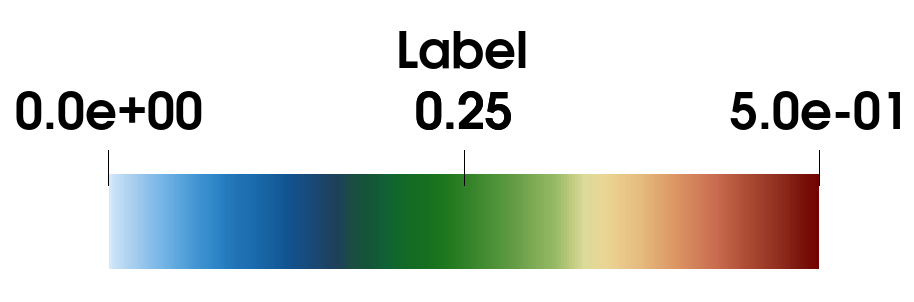}
\end{minipage}
\end{center}
   \caption{\textbf{Initial conditions.} The two squares on the left depict the initial stressed population density in zones $\Omega_1$ and $\Omega_2$, respectively, while the two squares on the right show the initial non-stressed population density in the same zones. In $\Omega_1$, the entire initial population is assumed to be stressed and is divided into three distinct clusters. Conversely, in $\Omega_2$, the entire population is non-stressed, also distributed into three separate clusters. This choice of initial conditions is justified by the fact that the disaster occurs only in $\Omega_1$.
}
\label{Initial_Conditions}
\end{center}
\end{figure}

\subsection{ A local control strategy to reduce stress levels in the network}
In this section, we propose implementing controls in specific areas to reduce stress levels in the network, with particular focus on \(\Omega_2\). Specifically, a first scenario with a control will be applied in the departure zone of \(\Omega_1\) and another scenario with a control in the arrival zone of \(\Omega_2\). The parameters considered in this section are the same as before, see Table  \ref{tab:parameters-systeme-ACP}. The control functions are defined as:

\begin{equation}
	\mathbf{u}_1(t,x) = K_1(t) \ \mathbf{p}_{21}(x) \ u_{P_1}, \quad x\in \Omega_1,
\end{equation}
in the departure area, and
\begin{equation}
	\mathbf{u}_2(t,x) = K_2(t) \ \varepsilon_{2}(y) m_{21} \int_{\Omega_1} \mathbf{p}_{21}(y)  \ u_{P_2}(t,y)dy, \quad y\in \Omega_2.
\end{equation}
in the arrival area, where \(K_1, K_2 : [0,+\infty) \longrightarrow [0,1]\) defined by
\begin{equation}\label{gamma} 
 K_i(t):=K_i \ \zeta(t,T_0,T_1), \quad i=1,2; \, 0\leq T_0 <T_1.
\end{equation}
where \(0 <K_1, K_2 \leq 1\) are the control rates representing the maximal proportions of stressed pedestrians who become unstressed by the control strategy. When \(K_1 = 0\) and \(K_2 = 0\), we revert to the main model. The function $ \zeta(t,T_0,T_1) $ is given by
\begin{equation}\label{z}
 \zeta(t,T_0,T_1) :=\left\{
 \begin{aligned}
 &0 & \quad \text{ if }  0\leq t<T_0 , \\
 & \dfrac{1}{2}-\dfrac{1}{2}\cos\left(\frac{t-T_0}{T_1-T_0}\pi\right) & \quad \text{ if } T_0\leq  t \leq T_1,   \\
 & 1  & \quad \text{ elswhere},    \\
  \end{aligned}
    \right.
\end{equation}
where \( T_0 \geq 0 \) is the initial time for implementing the local control strategy and \( 0 <T_1 \) is the maximum duration (end time) during which the control strategy is at its peak, i.e., \( K_i(t) = K_i \) for \( t \geq T_1 \) and \( i = 1, 2 \). The time interval \([T_0, T_1]\) represents the period during which the control strategy is active, aiming to reduce stress levels in this area to the maximum extent.
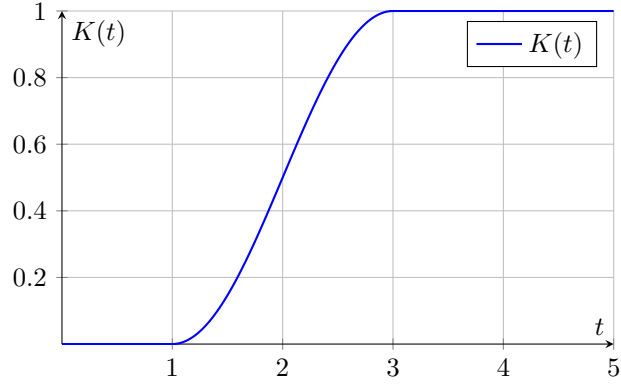
\begin{figure}[h]
		\centering
		\begin{tikzpicture}
			\begin{axis}[
				xlabel=$t$,
				ylabel=$K(t)$,
				legend pos=north east,
				grid=both,
				axis lines=middle,
				width=0.6\textwidth,
				height=6cm,
				samples=100,
				domain=0:5
				]
				\addplot[blue, thick, domain=0:1] {0};
				\addplot[blue, thick, domain=1:3] {0.5 - 0.5 * cos(deg(pi*(x-1)/2))};
				 \addplot[blue, thick, domain=3:5] {1};
				\legend{$K(t)$}
			\end{axis}
		\end{tikzpicture}
		\caption{Example of representative graph of the control function $K(t)=K\ \zeta(t,T_0,T_1)$ with $K=1 $, $T_0=1$ and $T=3$.}
		\label{fig: xi}
	\end{figure}
Applying these control functions to \(\Omega_1\) and \(\Omega_2\) respectively, the resulting controlled model is as follows:
 
\begin{equation}\label{control Model}
	\begin{aligned}
		\begin{cases}
			\begin{cases}
				\dfrac{\partial u_{P_1}}{\partial t} =& d_{P_1}\Delta u_{P_1} -\nabla \cdot (\vec{v}_{P_1}(u_{P_1})  u_{P_1}) + a_1\  u_{N_1} - b_1 u_{P_1} + f_1(u_{P_1}, u_{N_1})  u_{N_1}  u_{P_1} \\ &+ f_{21}(x,u_{P_1}, u_{P_2})-\mathbf{u}_1 (t,x),\quad \quad \quad \quad   \quad\quad \quad \quad \quad t>0, x \in \Omega_1, \\[4mm]
				\dfrac{\partial u_{N_1}}{\partial t} =& d_{N_1}\Delta u_{N_1} - \nabla \cdot (\vec{v}_{N_1}  u_{N_1}) - a_1  \ u_{N_1} + b_1 u_{P_1} - f_1(u_{P_1}, u_{N_1}) u_{N_1} u_{P_1} \\&+ f_{21}(x,u_{N_1}, u_{N_2})+\mathbf{u}_1(t,x),\quad \quad \quad \quad \quad \quad \quad \quad \quad t>0, x \in \Omega_1, \\[4mm]
				\partial_{\vec{\eta}} u_{P_1} =& \partial_{\vec{\eta}} u_{N_1} = 0, \quad \quad t>0, x \in \partial \Omega_1, \\\\
			\end{cases}\\[4mm]
			\begin{cases}
				\dfrac{\partial u_{P_2}}{\partial t} =& d_{P_2}\Delta u_{P_2} -\nabla \cdot (\vec{v}_{P_2}(u_{P_2}) u_{P_2}) + a_2 u_{N_2} - b_2 u_{P_2} + f_2(u_{P_2}, u_{N_2}) u_{N_2} u_{P_2} \\ &+ f_{12}(y,u_{P_2}, u_{P_1}) - \mathbf{u}_2(t,y),\quad \quad \quad \quad \quad \quad \quad t>0, x \in \Omega_2, \\[4mm]
				\dfrac{\partial u_{N_2}}{\partial t} =& d_{N_2}\Delta u_{N_2} - \nabla \cdot (\vec{v}_{N_2}  u_{N_2}) - a_2 u_{N_2} + b_2 u_{P_2} - f_2(u_{P_2}, u_{N_2}) u_{N_2} u_{P_2} \\ &+ f_{12}(y,u_{N_2}, u_{N_1})+\mathbf{u}_2(t,y), \quad \quad \quad \quad \quad \quad \quad t>0, y \in \Omega_2, \\[4mm]
				\partial_{\vec{\eta}} u_{P_2} =& \partial_{\vec{\eta}} u_{N_2} = 0, \quad \quad t>0, y \in \partial \Omega_2. \\\\
			\end{cases}
		\end{cases}
	\end{aligned}
\end{equation}

\noindent \textbf{Scenario 1 (Sc1): A local control strategy in the departure area to reduce stress levels in $\Omega_1$ and its impact on $\Omega_2$.}
In this first scenario, we consider only the control \(\mathbf{u}_1\) which acts on the stressed population in the departure area of \(\Omega_1\). We assume that \(\mathbf{u}_2 = 0\) in \eqref{control Model}.

\noindent\textbf{Scenario 2 (Sc2): A local control strategy in the arrival area to reduce stress levels in $\Omega_2$.}
In this second scenario, we consider this time only the control \(\mathbf{u}_2\) which acts on the stressed population in the arrival area of \(\Omega_2\). We assume that \(\mathbf{u}_1 = 0\) in \eqref{control Model}.\\

In the sequel, we give our numerical results showing the spatiotemporal evolution of our unidirectional model case, as well as comparison between this model case with respect to the two control scenarios.
\begin{figure}[h!]
	\centering
	\begin{tabular}{cc}
 		\subcaptionbox{Time evolution of the total mass $ M_{P_1}(t)=\int_{\Omega_1} u_{P_1}(t,x) dx $ in the zone $\Omega_1$.}{\includegraphics[width=0.5\linewidth, height=0.18\textheight]{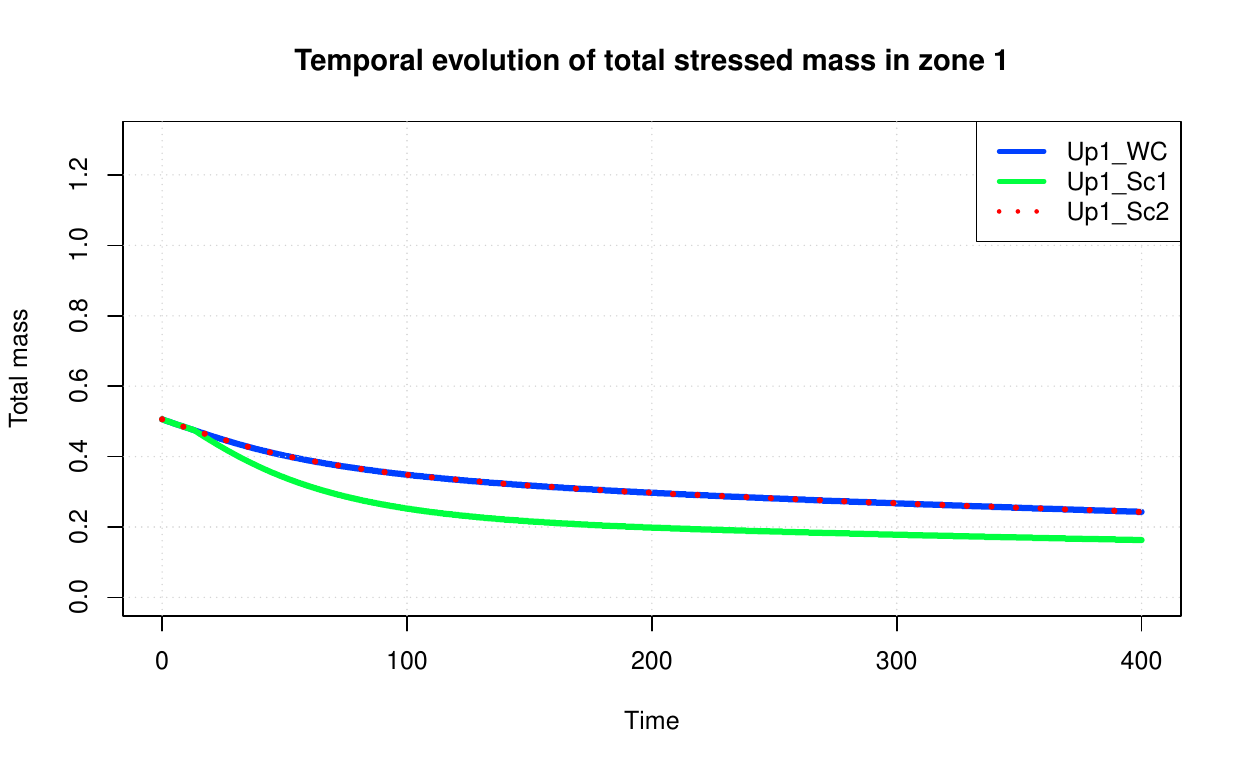}\label{pop-1}} & 
		\subcaptionbox{Time evolution of the total mass $ M_{P_2}(t)=\int_{\Omega_2} u_{P_2}(t,x) dx $ in the zone $\Omega_2$.}{\includegraphics[width=0.5\linewidth, height=0.18\textheight]{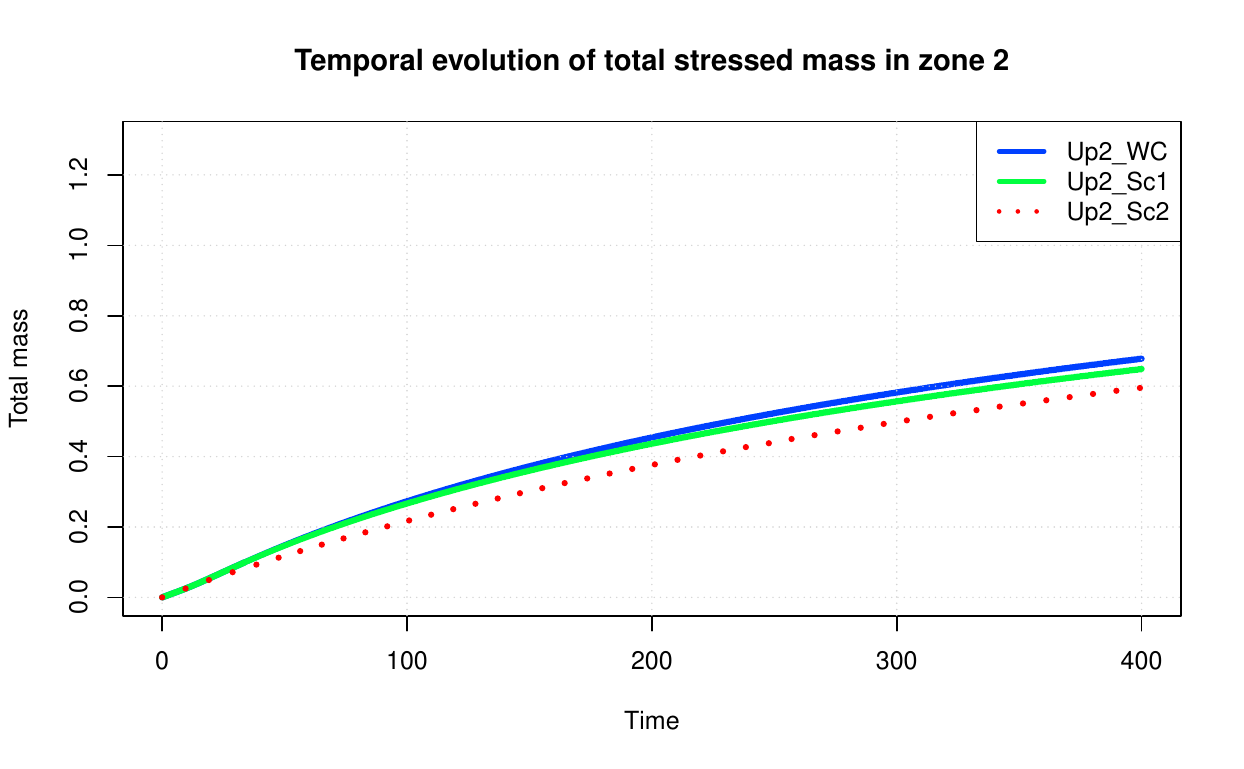}} 
	\end{tabular}
	\caption{Temporal evolution of stressed populations on both zones $\Omega_1$ (left) and $\Omega_2$ (right) for the three model cases (Wc), (Sc1) and (Sc2). The density in blue (Wc) represents the situation without control, the density in green represents the situation with only control $\bold{u}_1$, and the density in dashed red represents the situation with control $\bold{u}_2$.}
	\label{pop-fig1}
\end{figure}

\begin{figure}[h!]
	\centering
	\begin{tabular}{cc}
		\subcaptionbox{Time evolution of total stressed mass $ M_{P}(t)=\int_{\Omega_1} u_{P_1}(t,x) dx + \int_{\Omega_2} u_{P_2}(t,x) dx $ in the whole network.}{\includegraphics[width=0.5\linewidth, height=0.18\textheight]{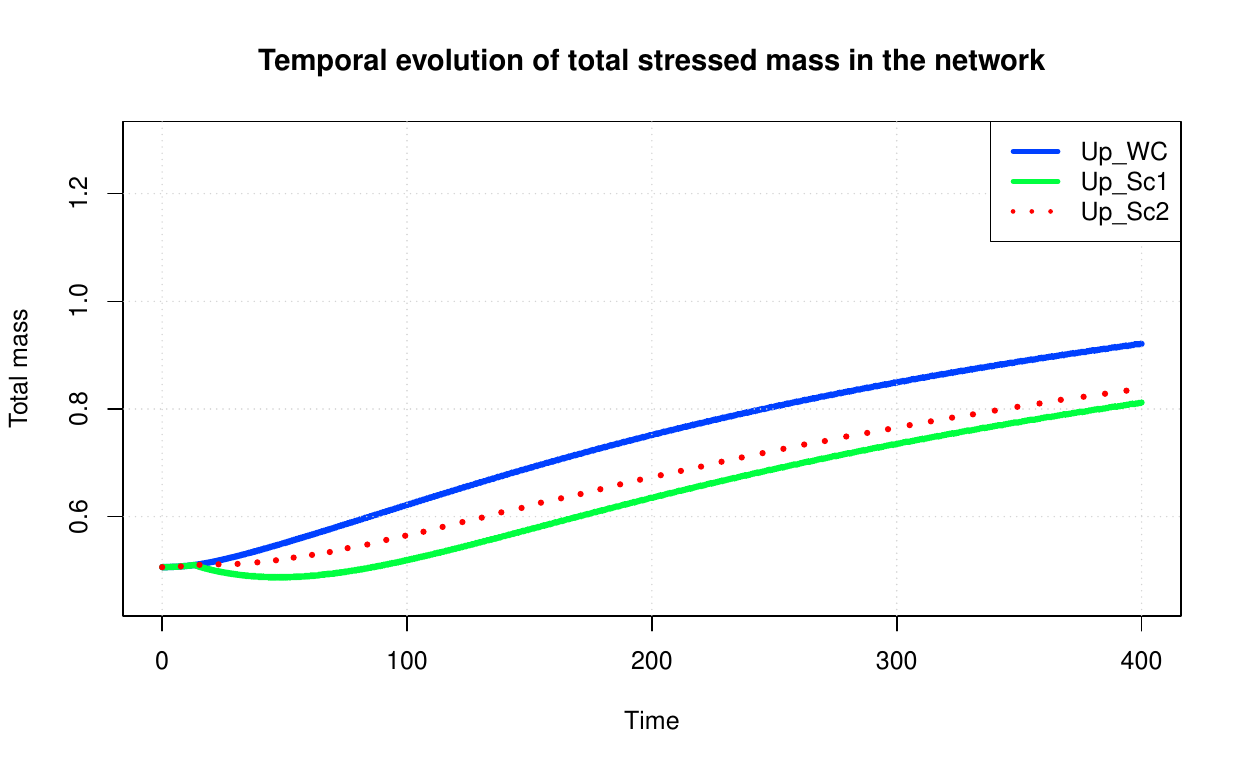}\label{pop}} & 
		\subcaptionbox{Time evolution of total stressed mass $ M_{N}(t)=\int_{\Omega_1} u_{N_1}(t,x) dx + \int_{\Omega_2} u_{N_2}(t,x) dx $ in the whole network.}{\includegraphics[width=0.5\linewidth, height=0.18\textheight]{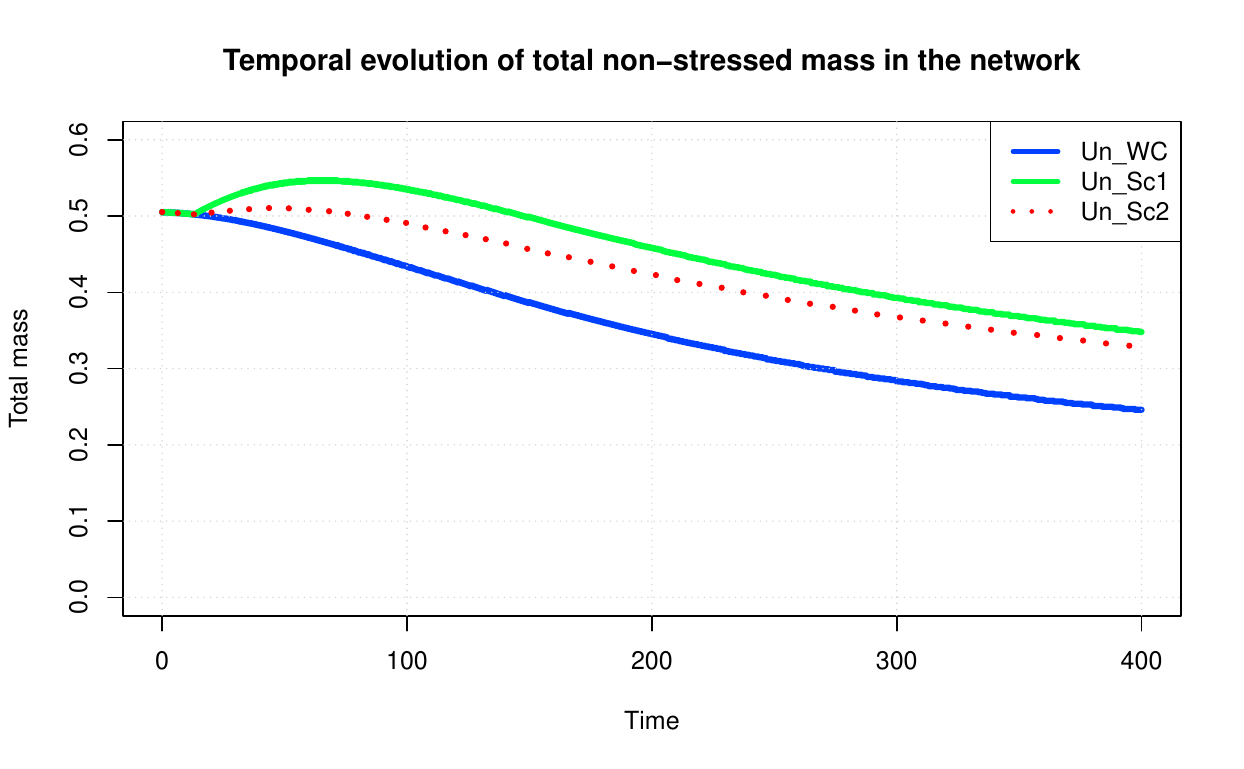}} 
	\end{tabular}
	\caption{Temporal evolution of non-stressed population on both zones $\Omega_1$ (left) and $\Omega_2$ (right) for the three model cases (Wc), (Sc1) and (Sc2). The density in blue (Wc) represents the situation without control, the density in green represents the situation with only control $\bold{u}_1$, and the density in dashed red represents the situation with control $\bold{u}_2$.}
	\label{pop-fig2}
\end{figure}

\begin{figure}[h!]
\begin{minipage}[c]{.32\linewidth}
\begin{center}
Without control (Wc)
\end{center}
\includegraphics[width=4cm,height=2cm]{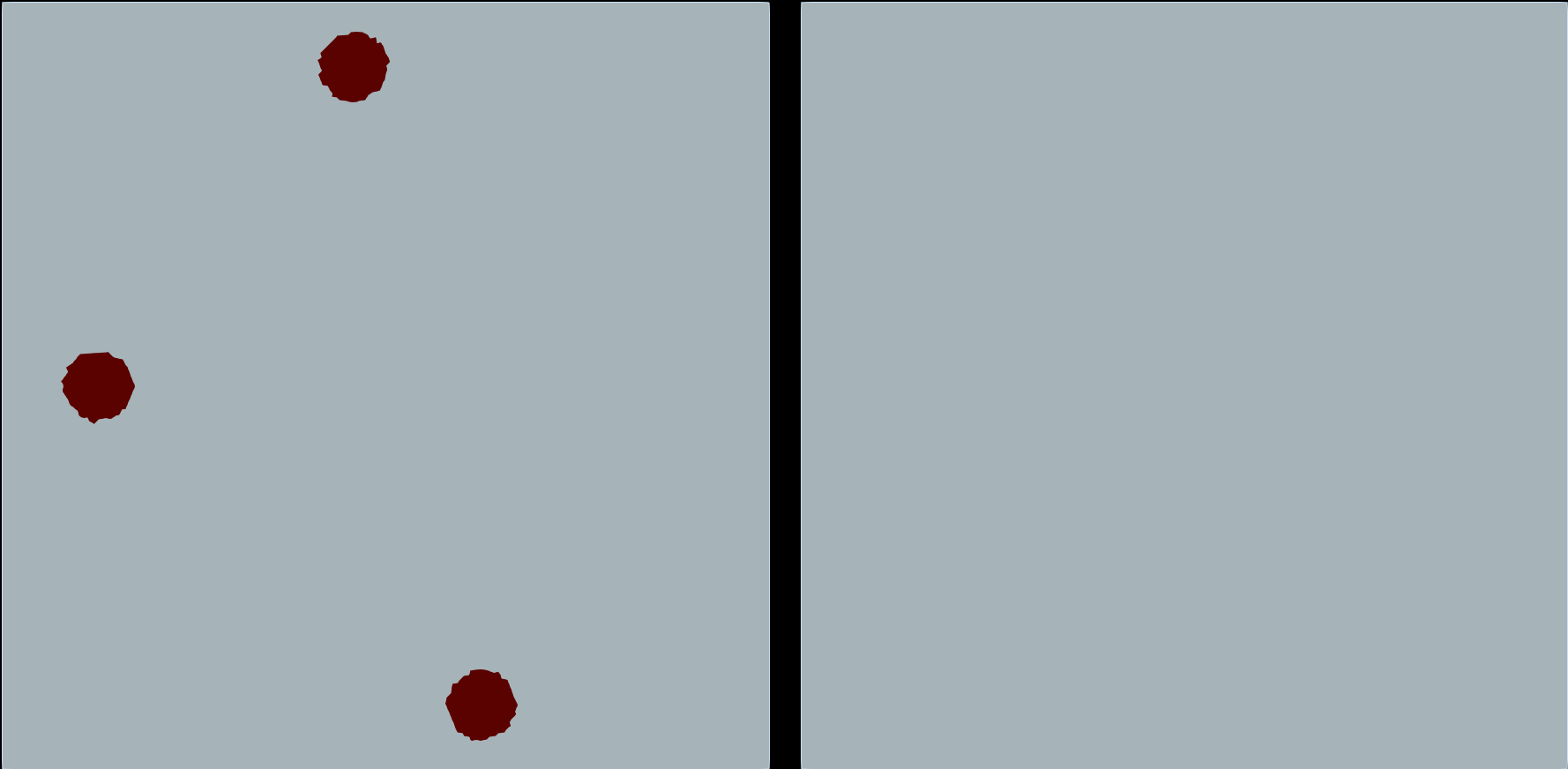}
\end{minipage}
\begin{minipage}[c]{.32\linewidth}
\begin{center}
First scenario (Sc1)
\end{center}
\includegraphics[width=4cm,height=2cm]{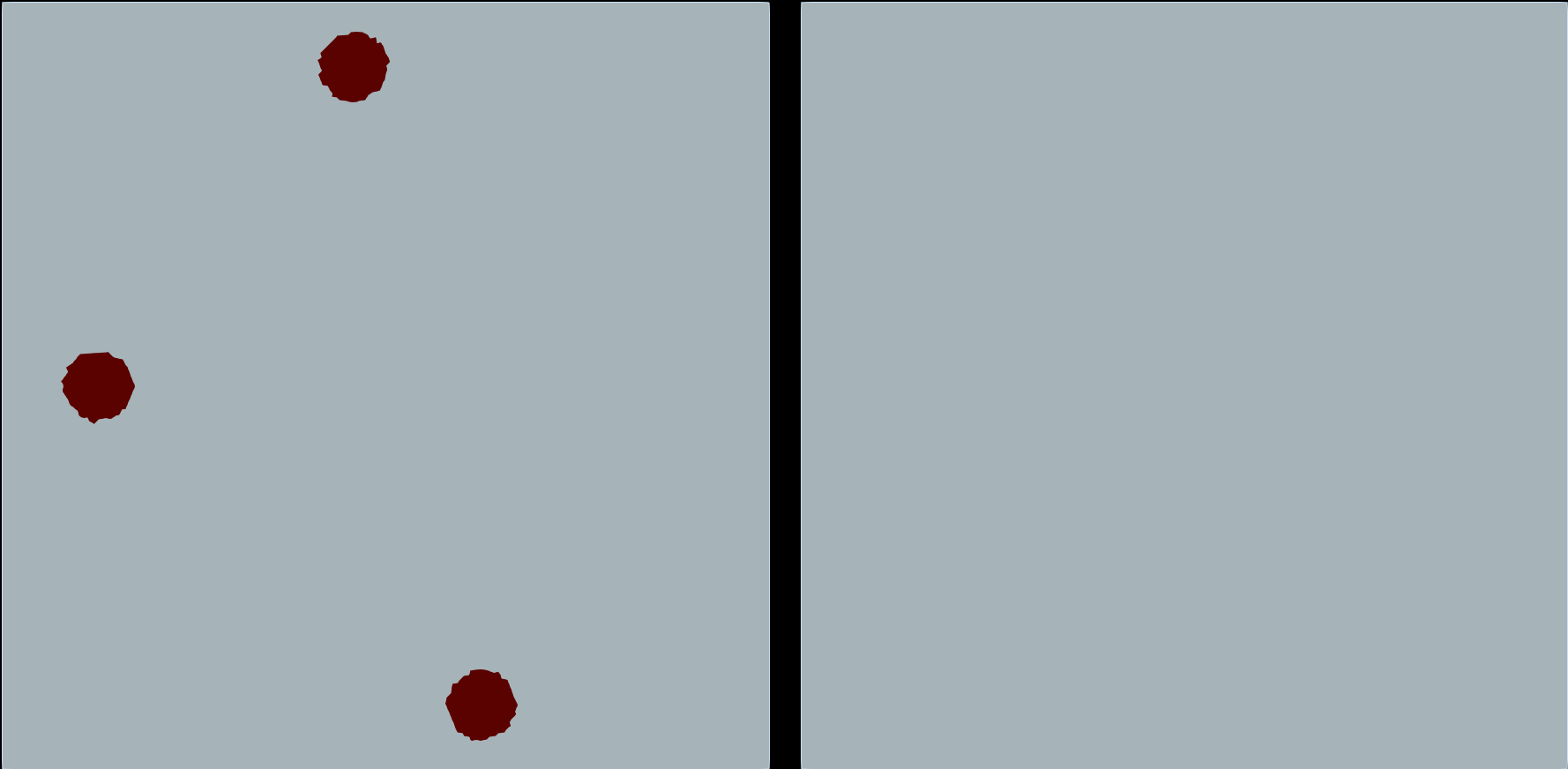}
\end{minipage}
\begin{minipage}[c]{.32\linewidth}
\begin{center}
Second scenario (Sc2)
\end{center}
\includegraphics[width=4cm,height=2cm]{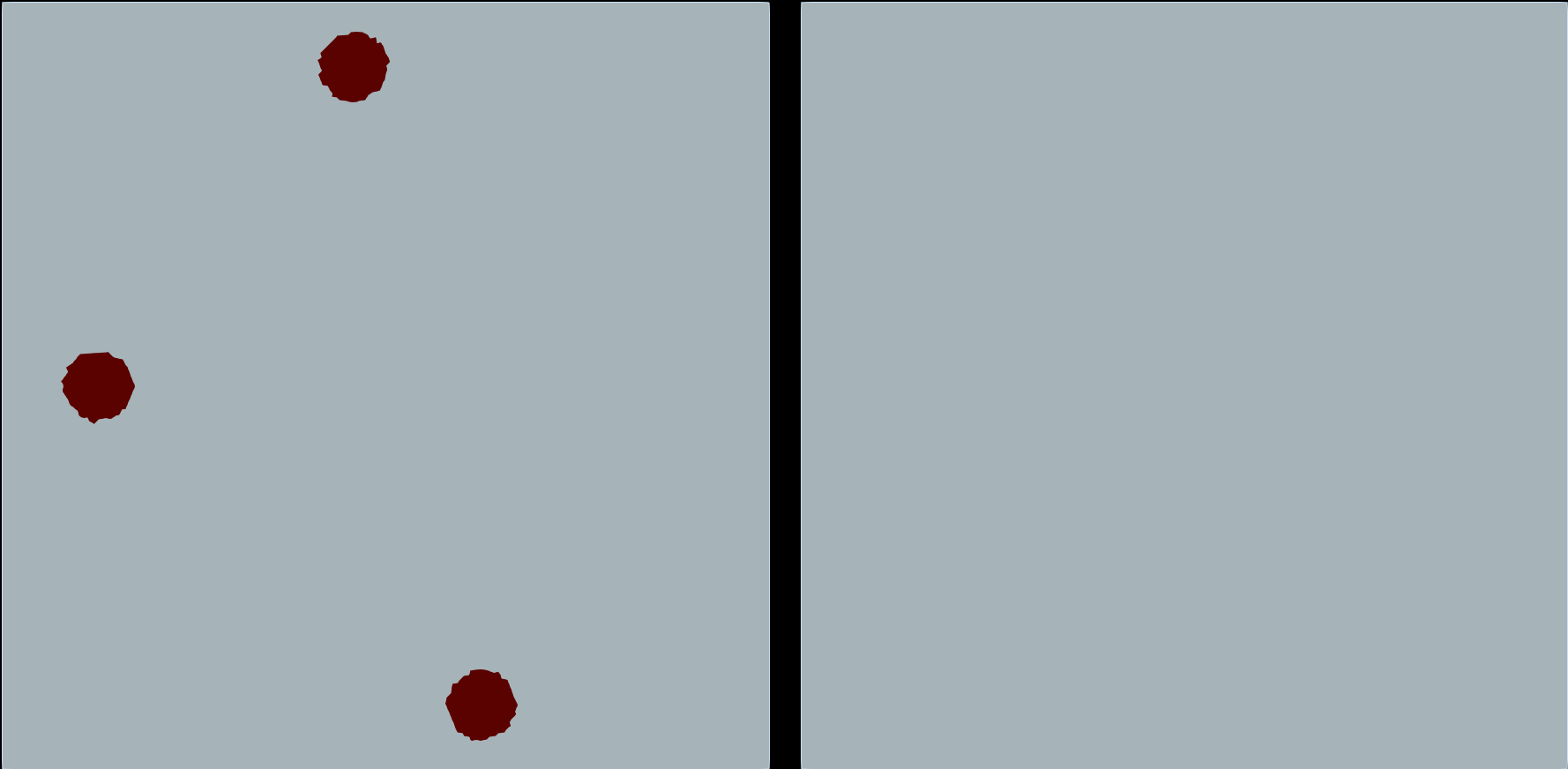}
\end{minipage}
\begin{center}
 $t=0$
\end{center}
\begin{minipage}[c]{.32\linewidth}
\includegraphics[width=4cm,height=2cm]{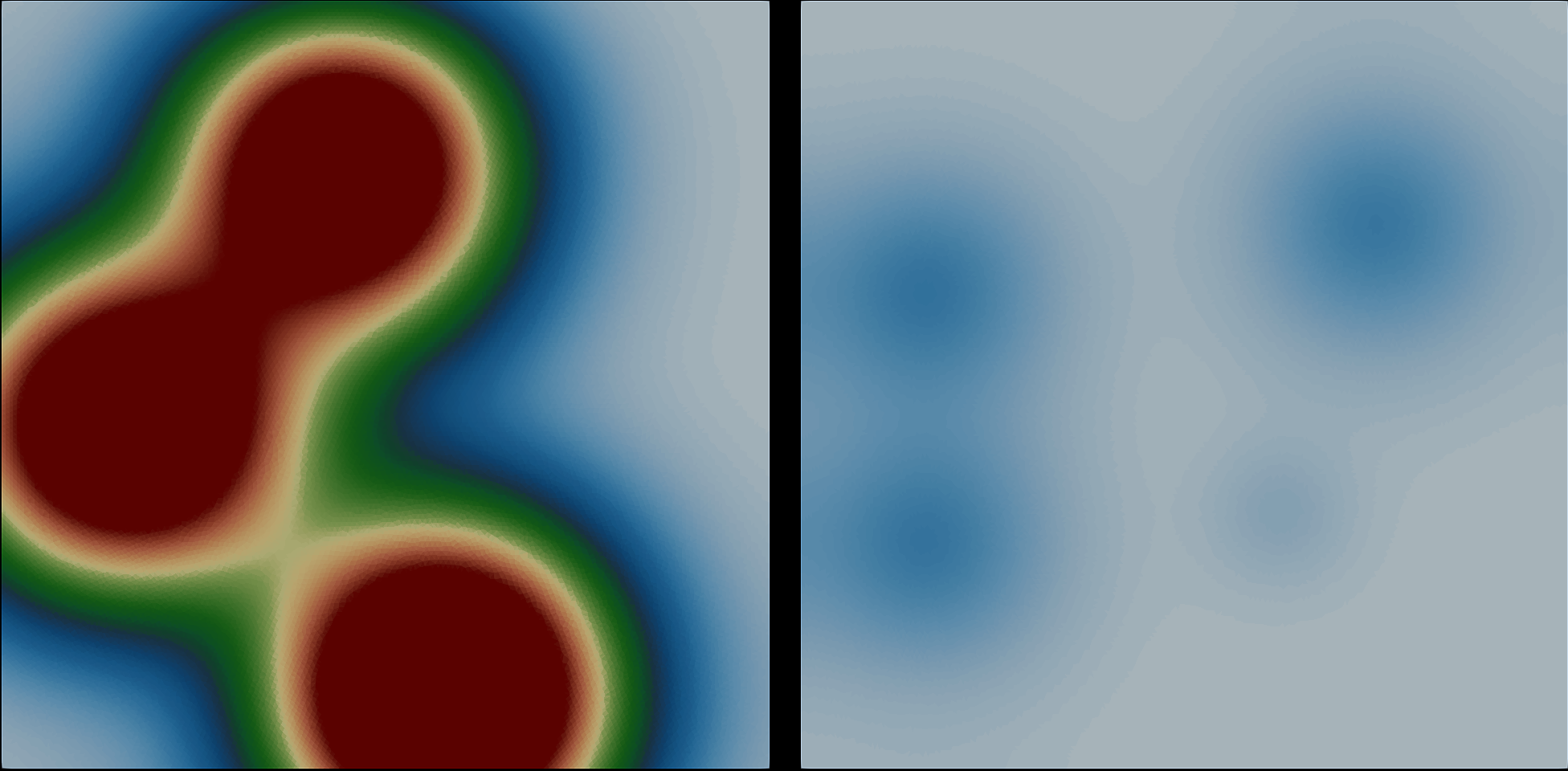}
\end{minipage}
\begin{minipage}[c]{.32\linewidth}
\includegraphics[width=4cm,height=2cm]{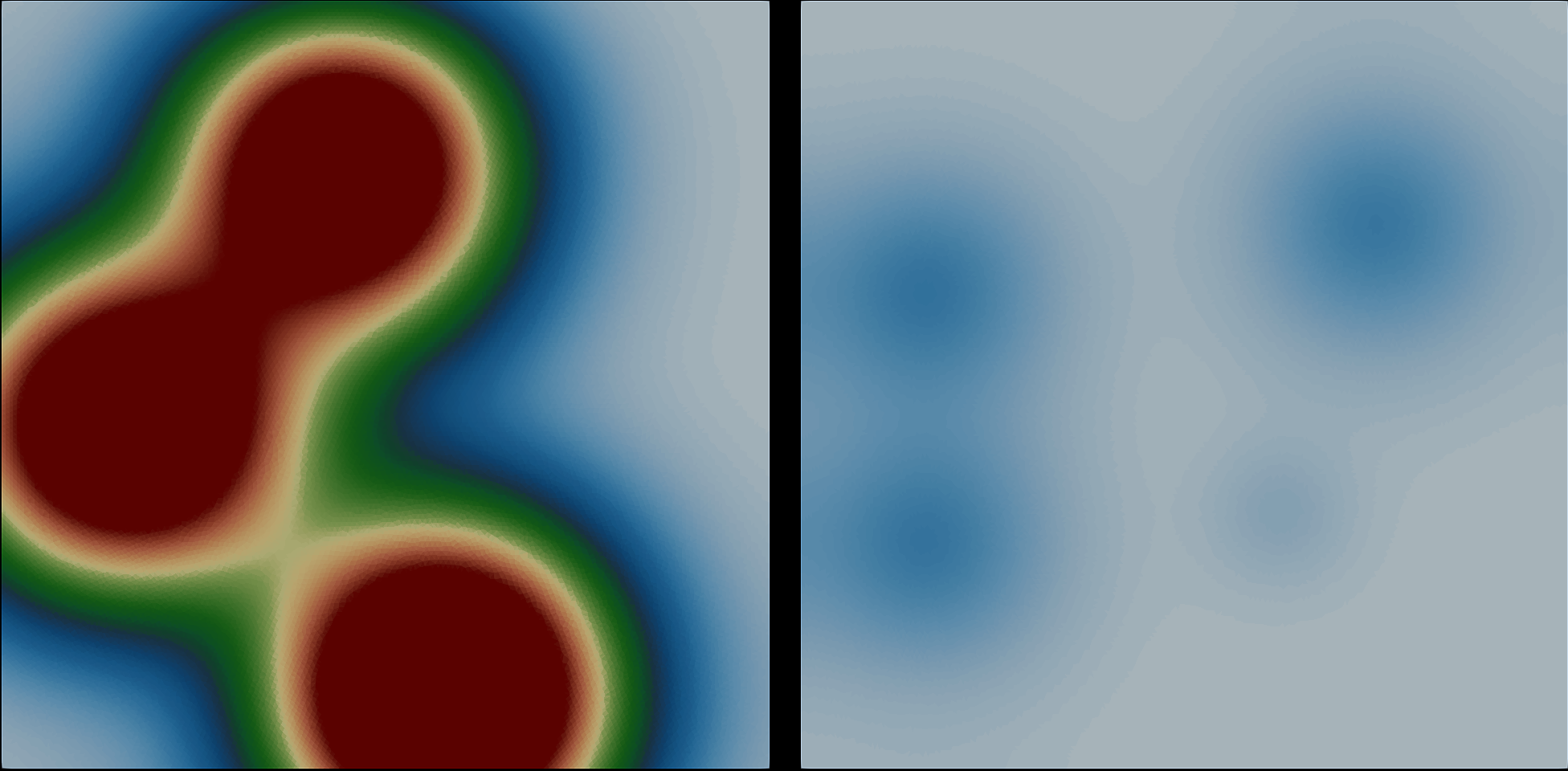}
\end{minipage}
\begin{minipage}[c]{.32\linewidth}
\includegraphics[width=4cm,height=2cm]{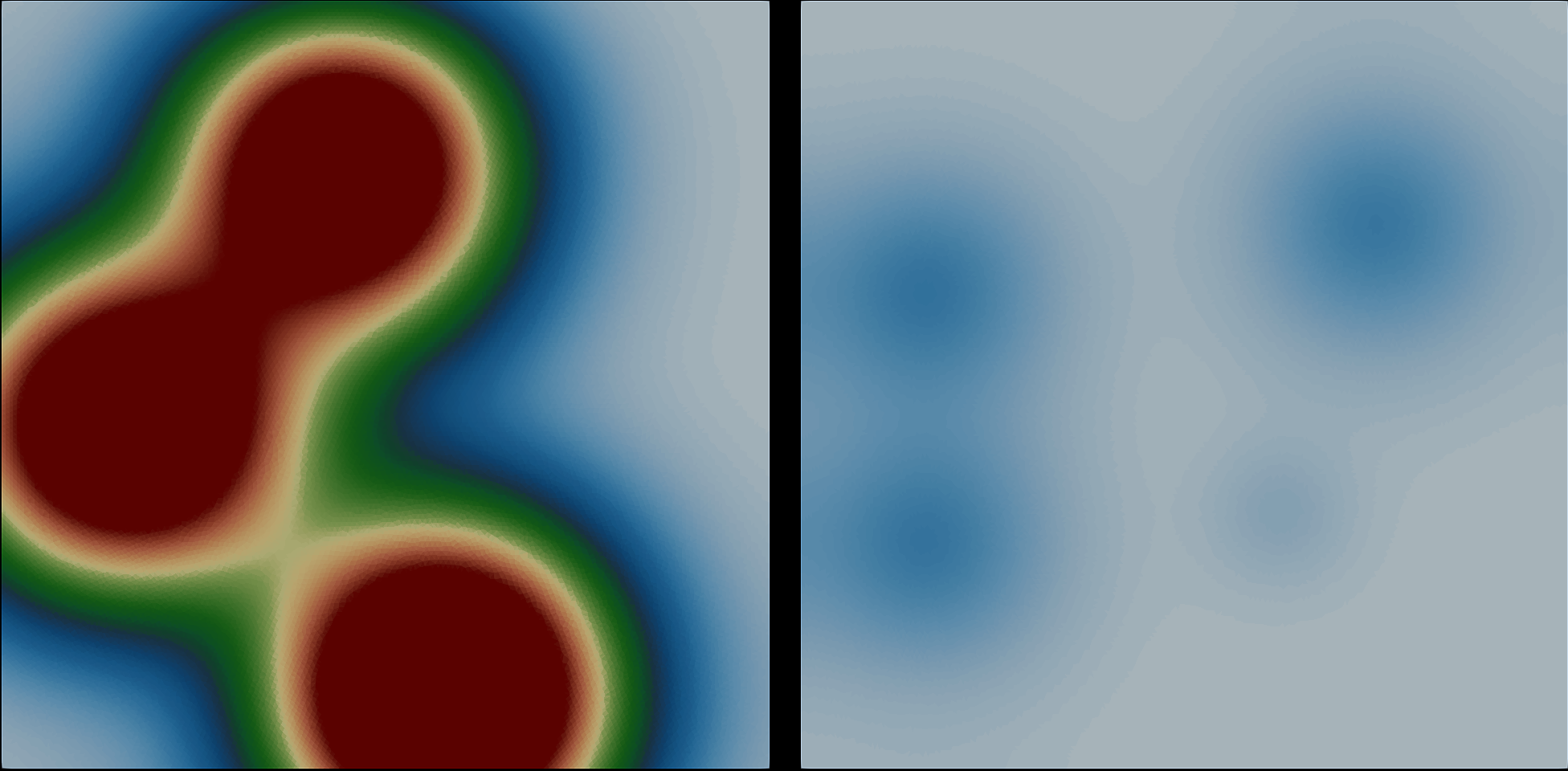}
\end{minipage}
\begin{center}
$t=10$
\end{center}
\begin{minipage}[c]{.32\linewidth}
\includegraphics[width=4cm,height=2cm]{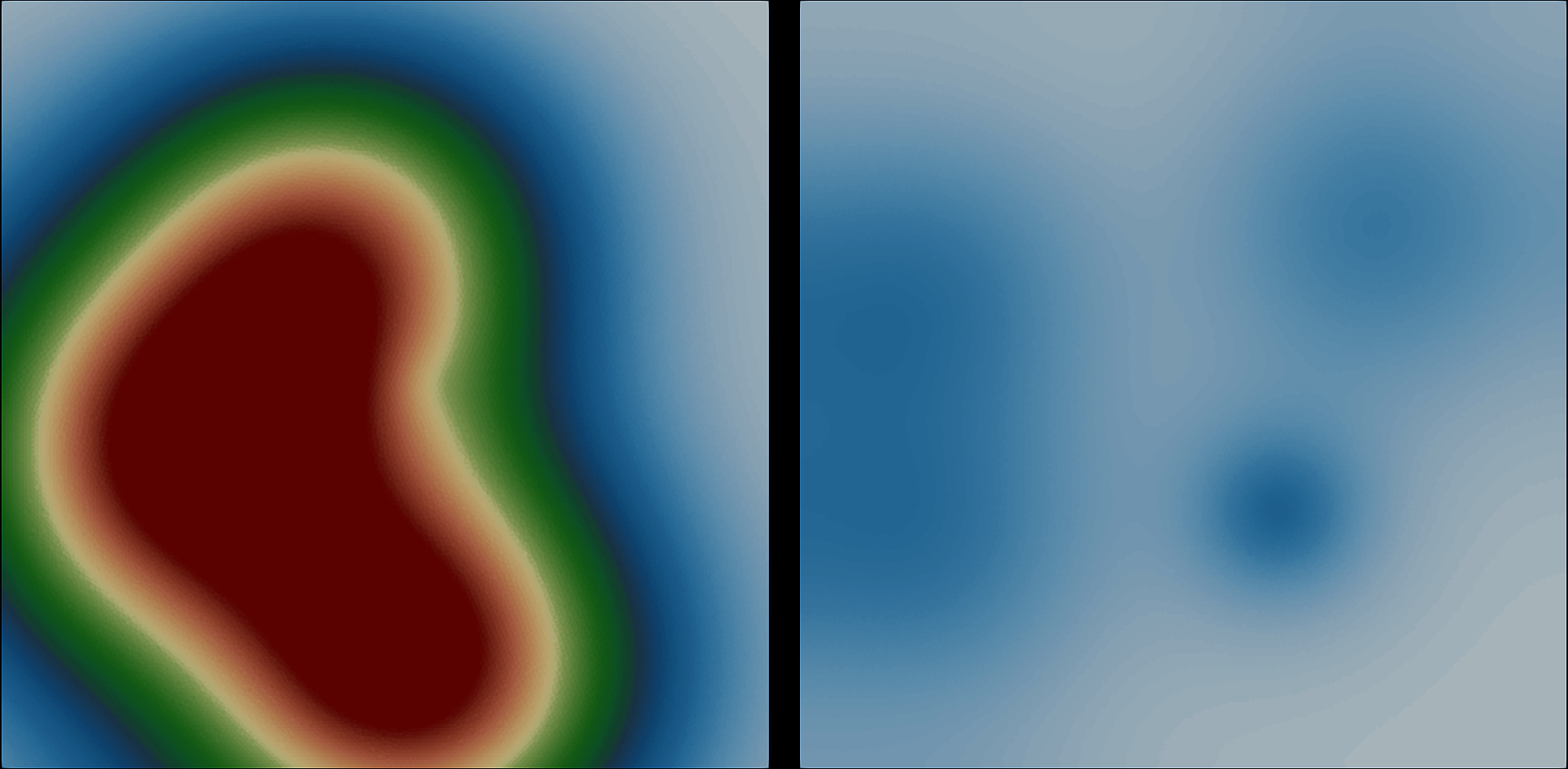}
\end{minipage}
\begin{minipage}[c]{.32\linewidth}
\includegraphics[width=4cm,height=2cm]{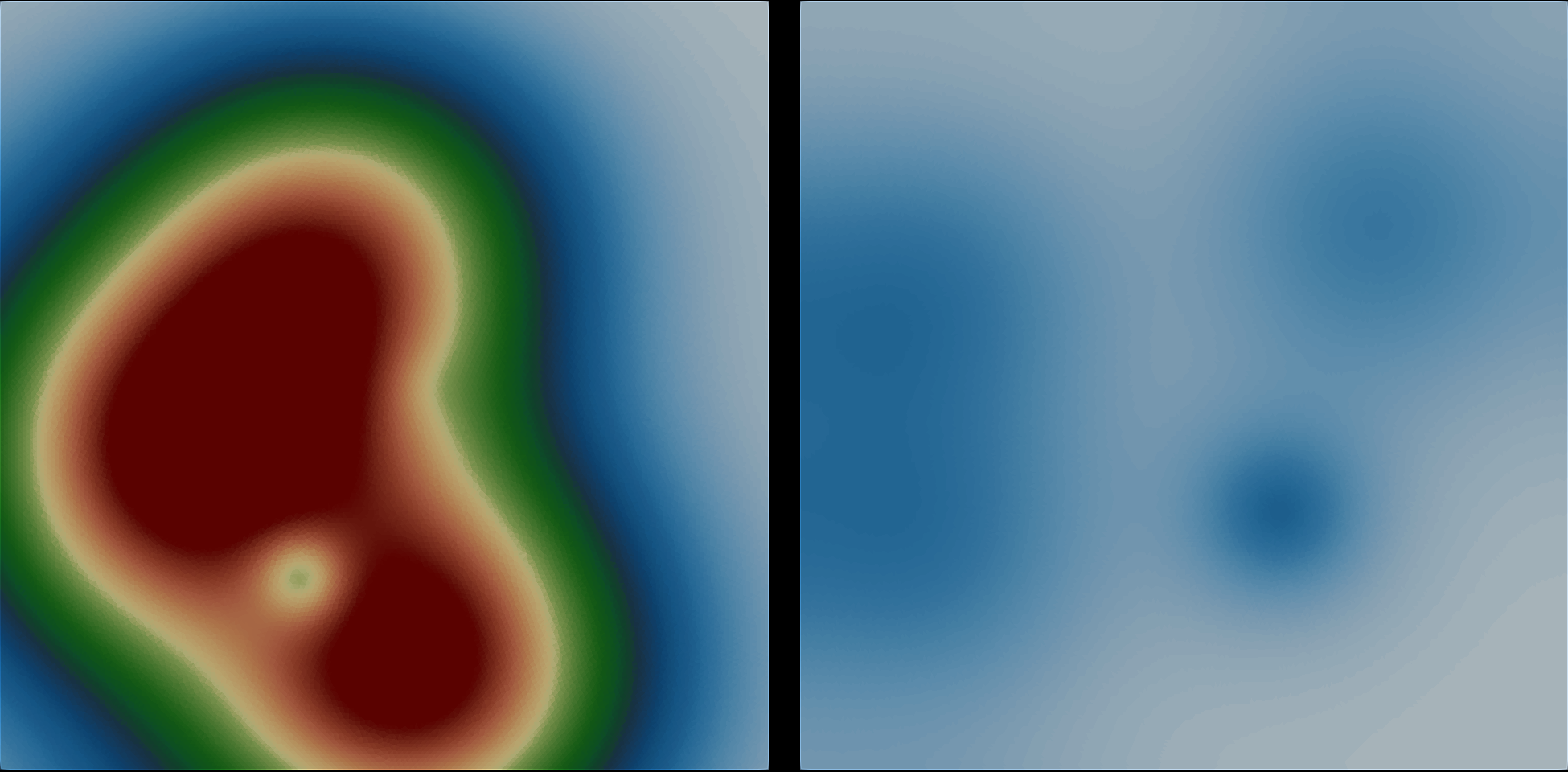}
\end{minipage}
\begin{minipage}[c]{.32\linewidth}
\includegraphics[width=4cm,height=2cm]{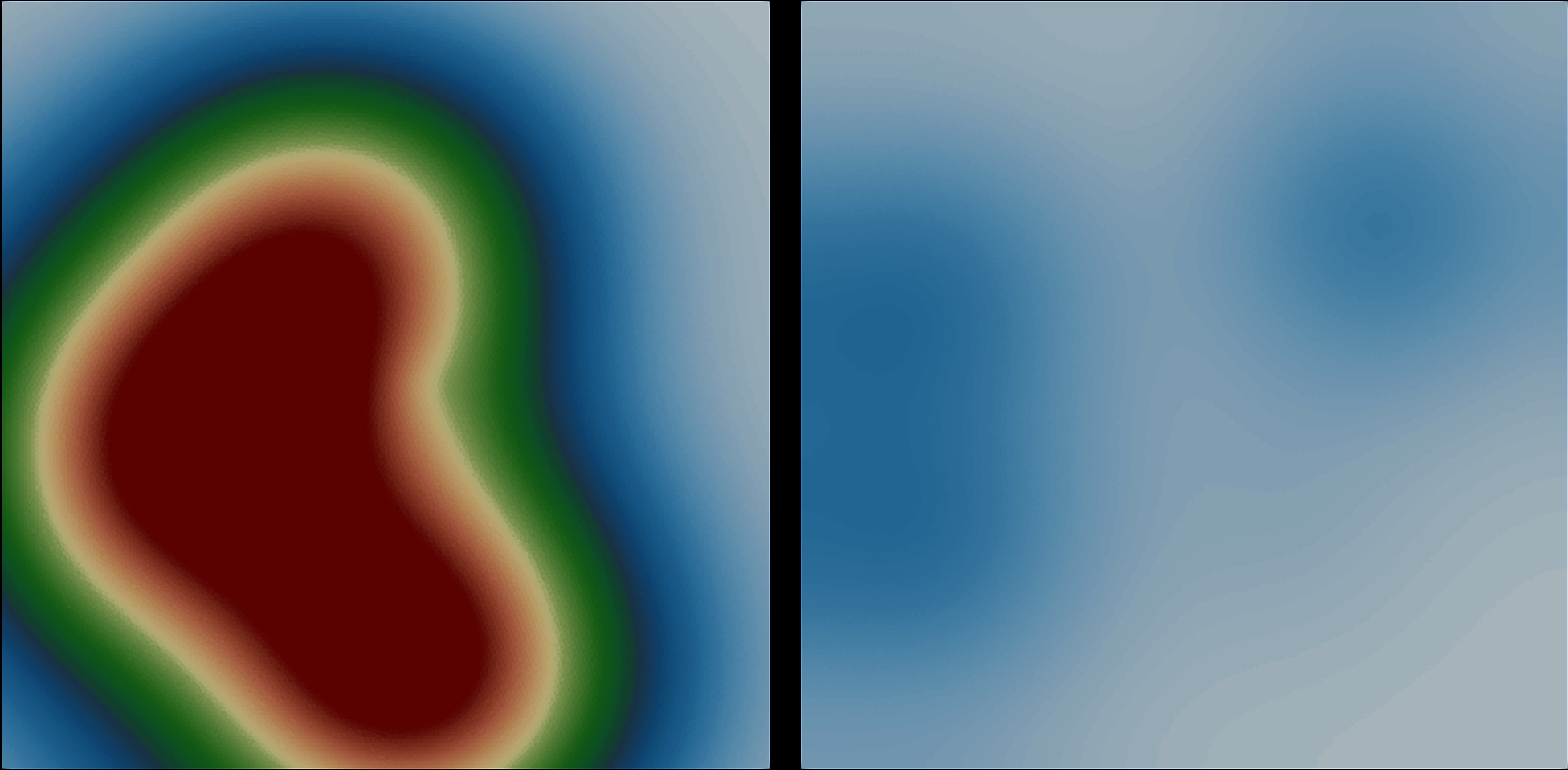}
\end{minipage}
\begin{center}
$t=20$
\end{center}
\begin{minipage}[c]{.32\linewidth}
\includegraphics[width=4cm,height=2cm]{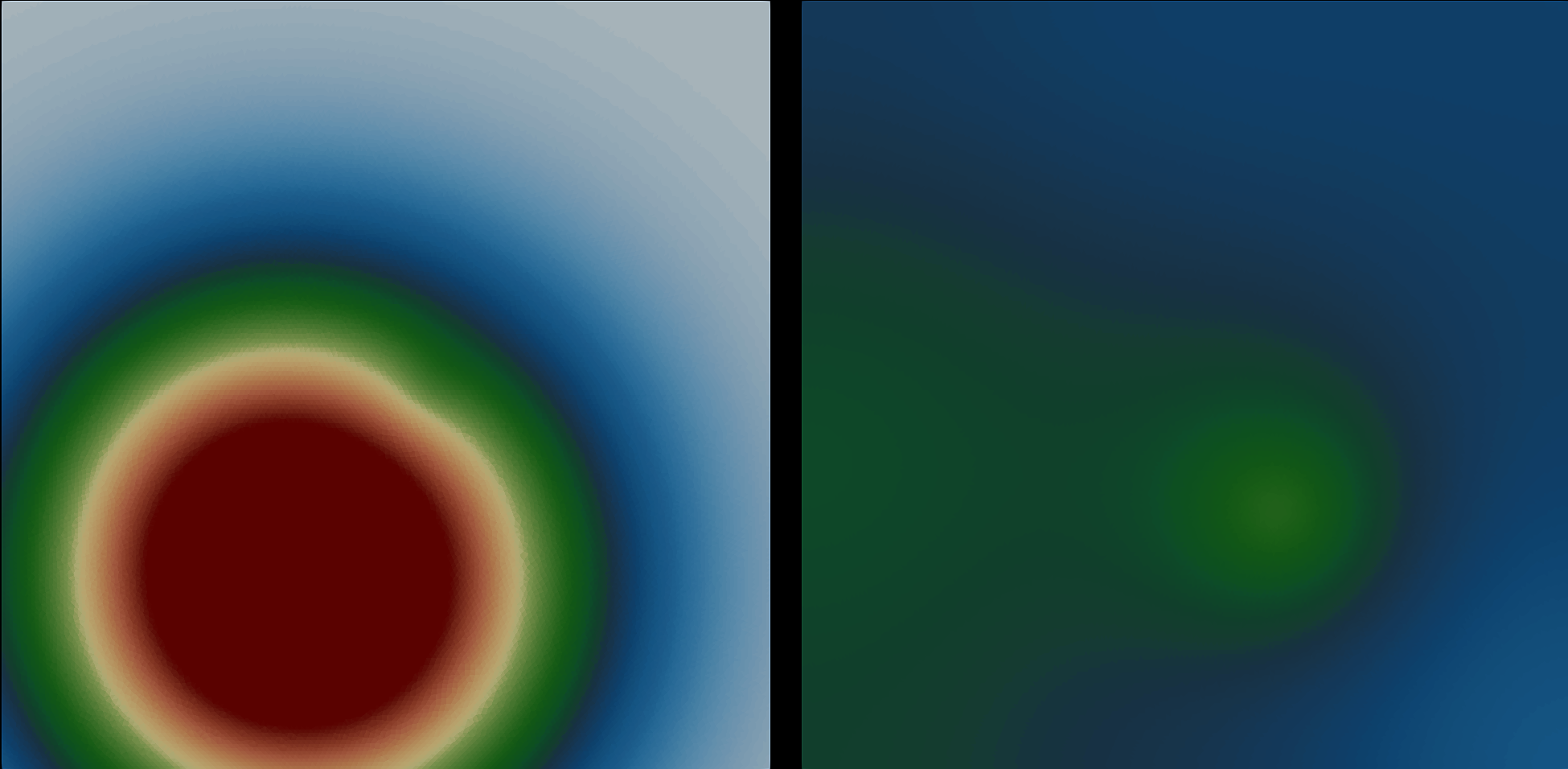}
\end{minipage}
\begin{minipage}[c]{.32\linewidth}
\includegraphics[width=4cm,height=2cm]{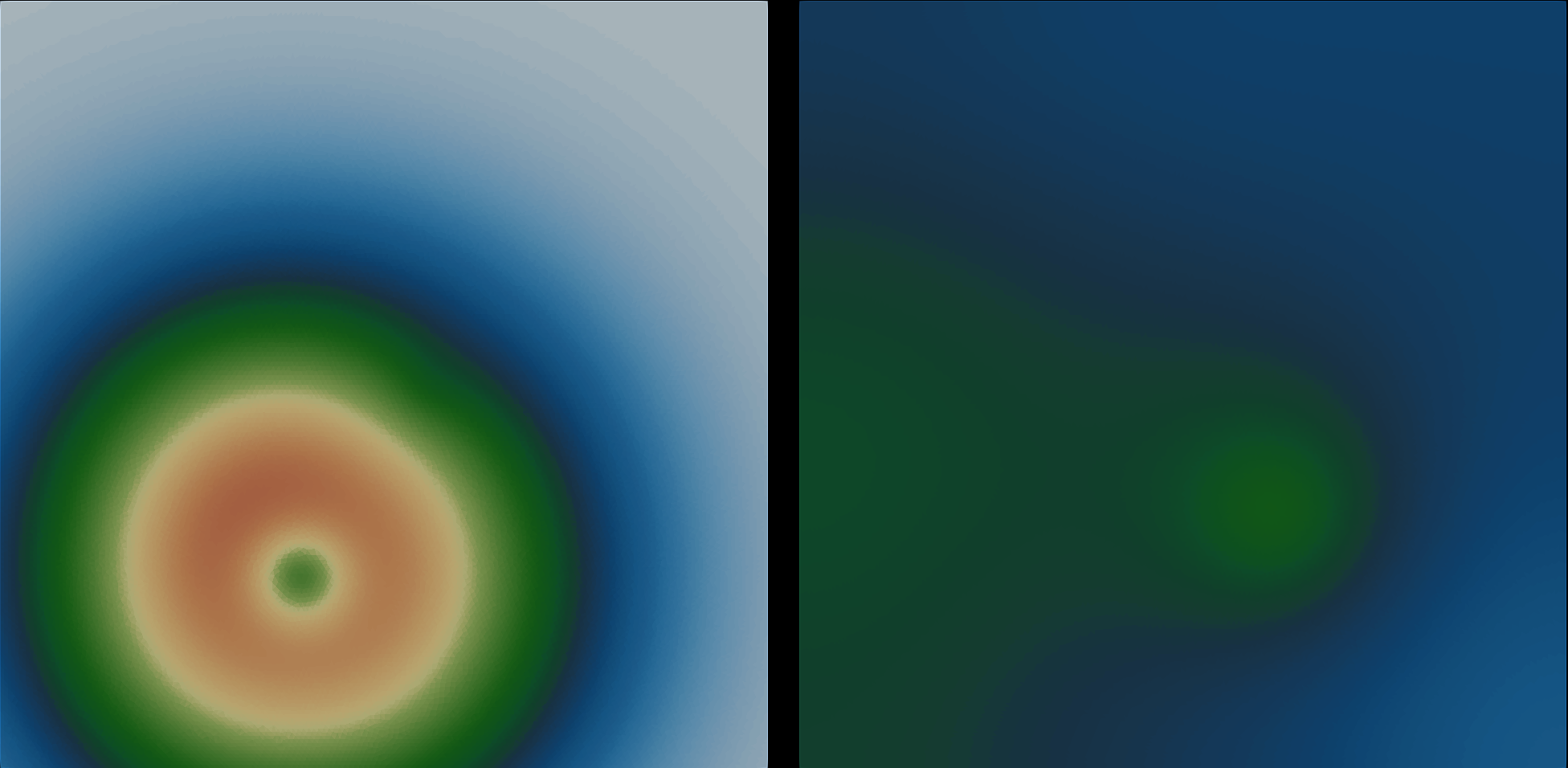}
\end{minipage}
\begin{minipage}[c]{.32\linewidth}
\includegraphics[width=4cm,height=2cm]{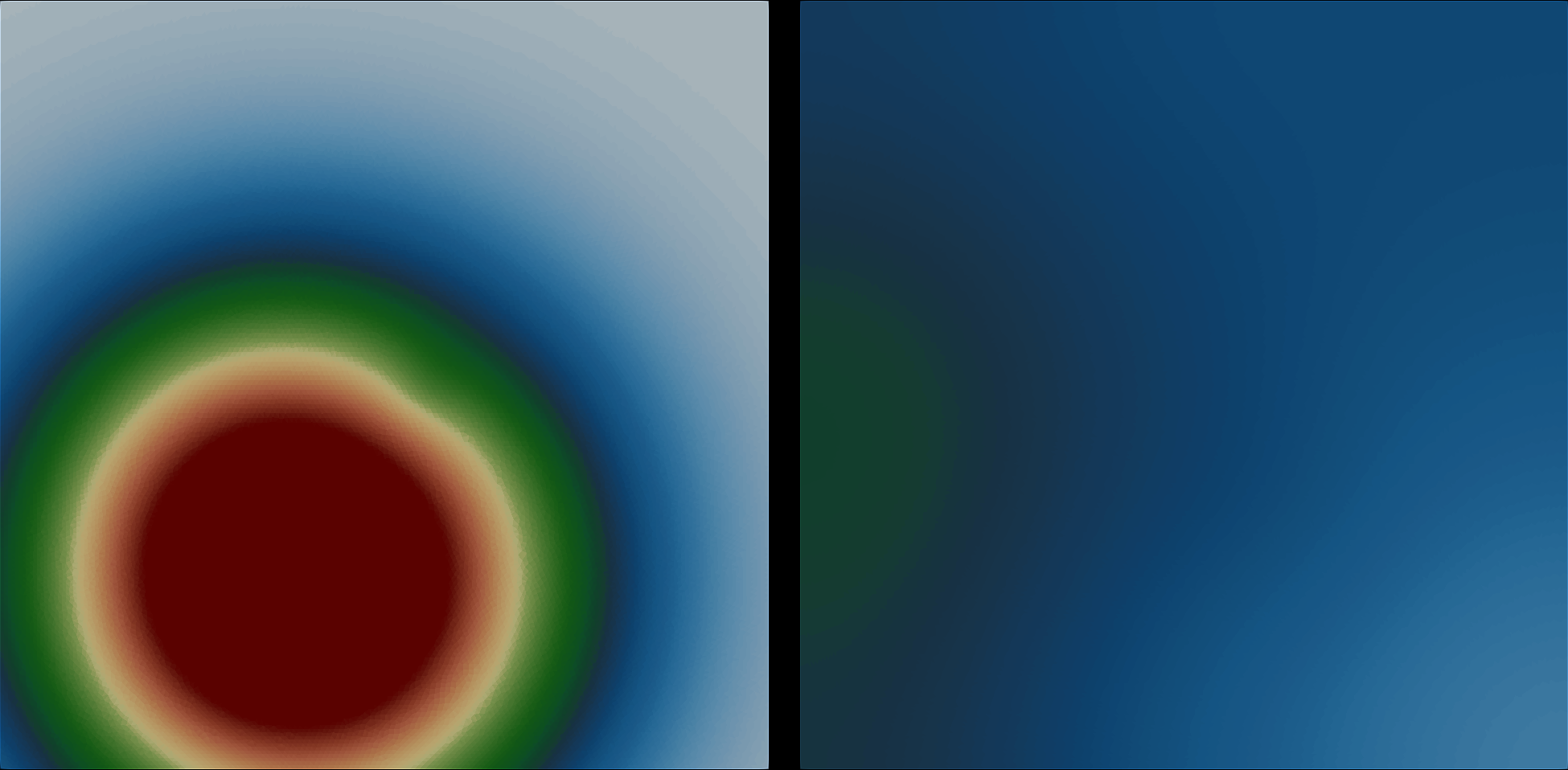}
\end{minipage}
\begin{center}
 $t=100$
\end{center}
\begin{minipage}[c]{.32\linewidth}
\includegraphics[width=4cm,height=2cm]{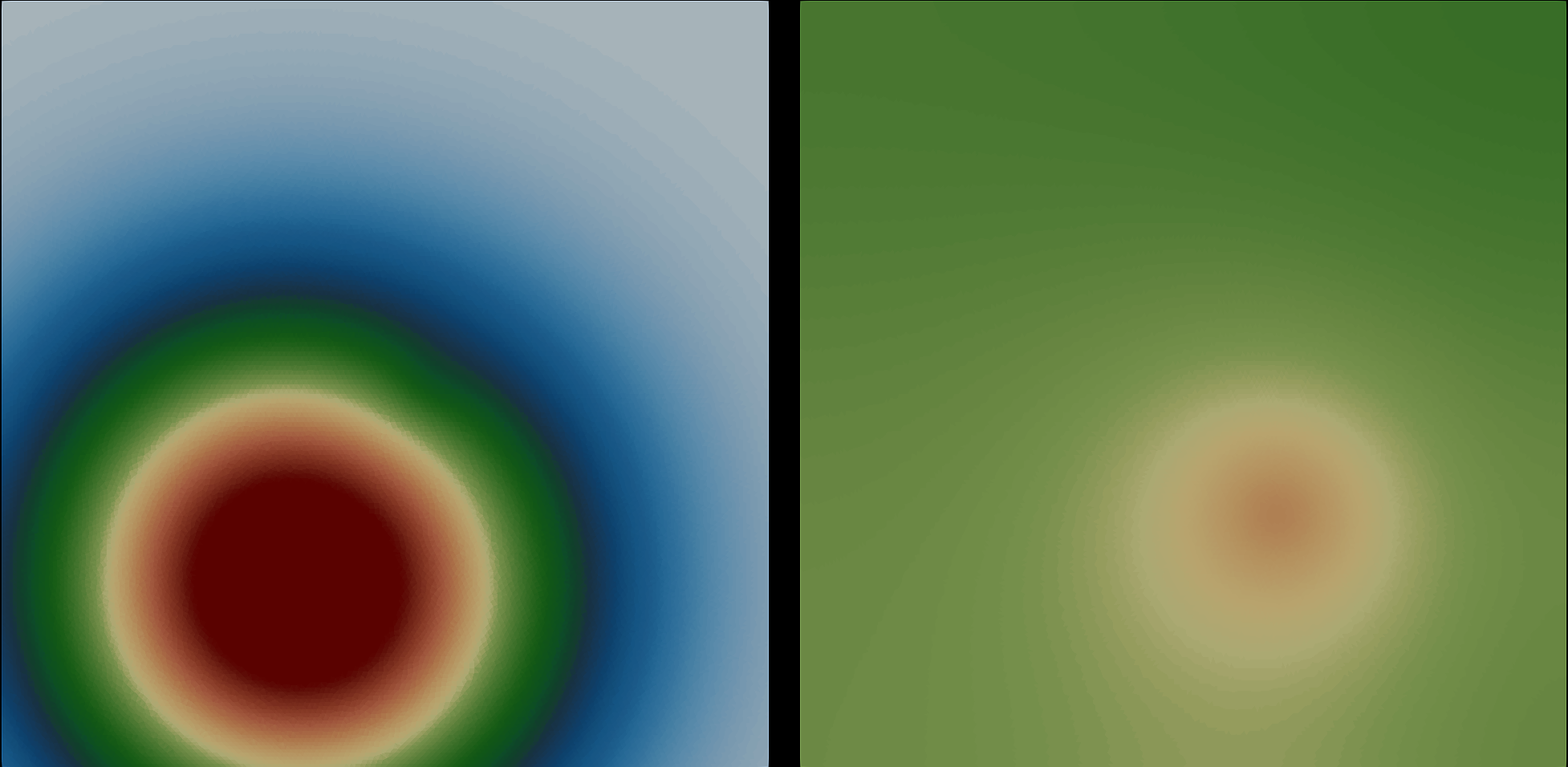}
\end{minipage}
\begin{minipage}[c]{.32\linewidth}
\includegraphics[width=4cm,height=2cm]{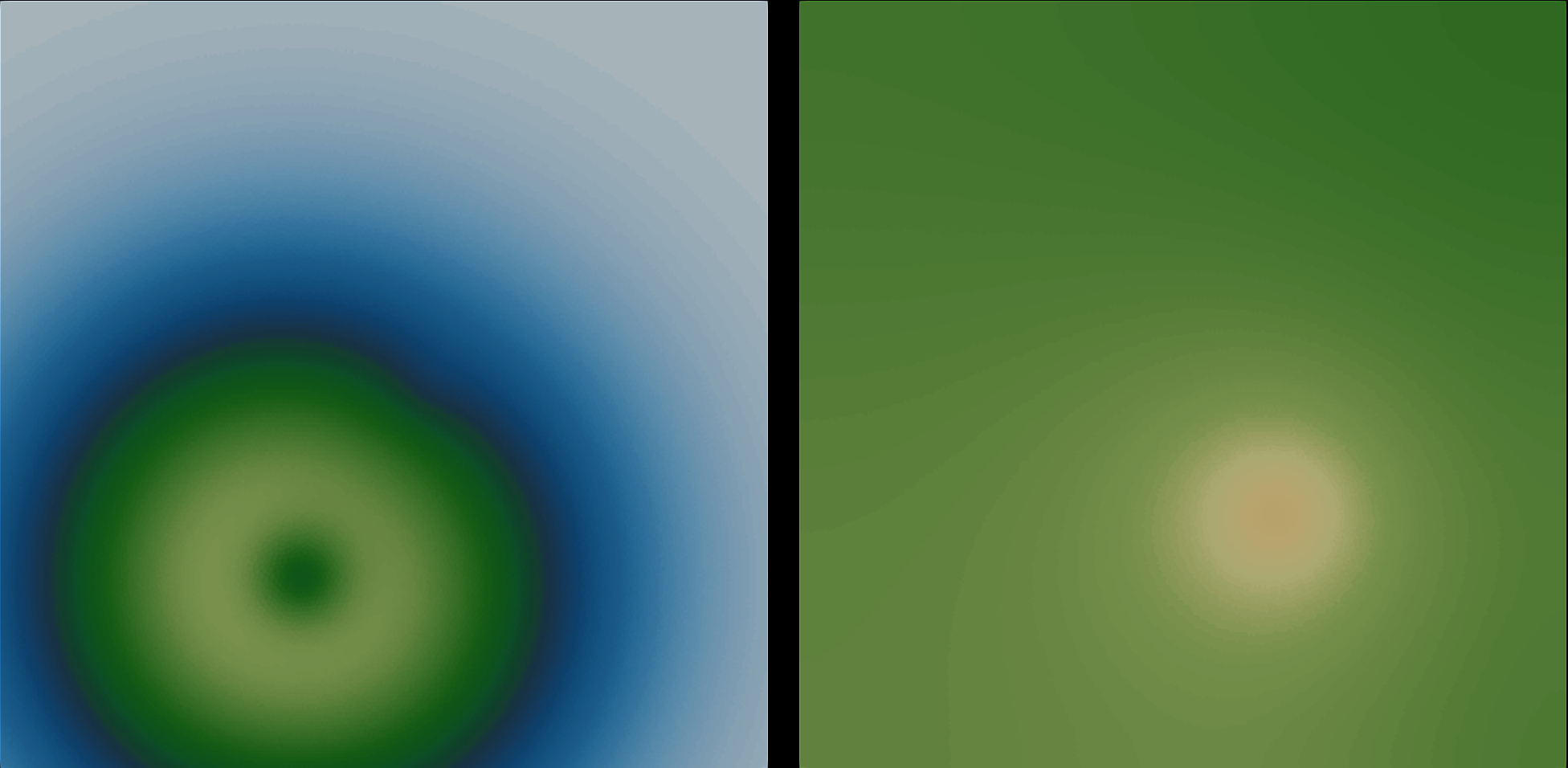}
\end{minipage}
\begin{minipage}[c]{.32\linewidth}
\includegraphics[width=4cm,height=2cm]{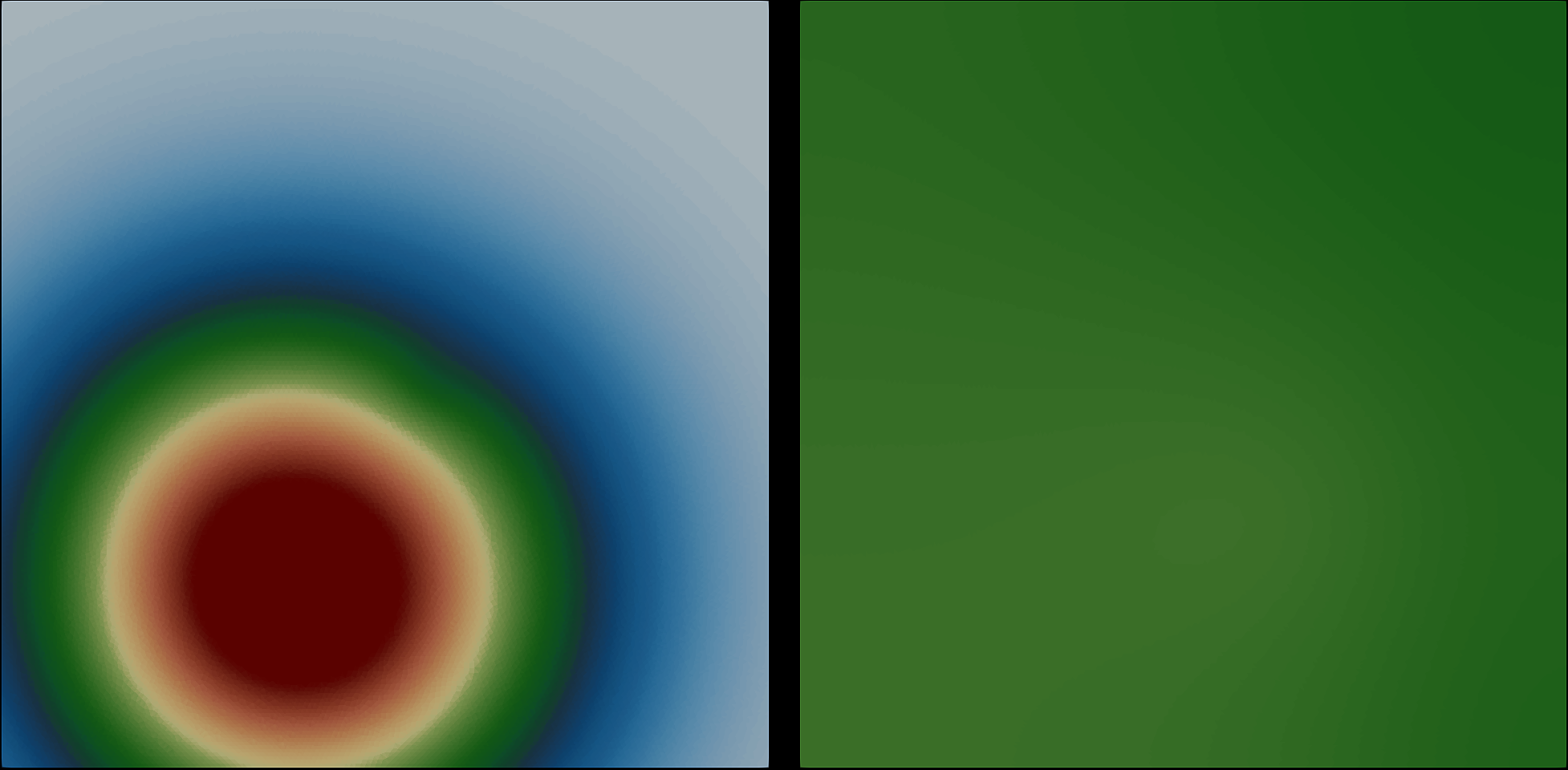}
\end{minipage}
\begin{center}
 $t=250$
\end{center}
\begin{minipage}[c]{.32\linewidth}
\includegraphics[width=4cm,height=2cm]{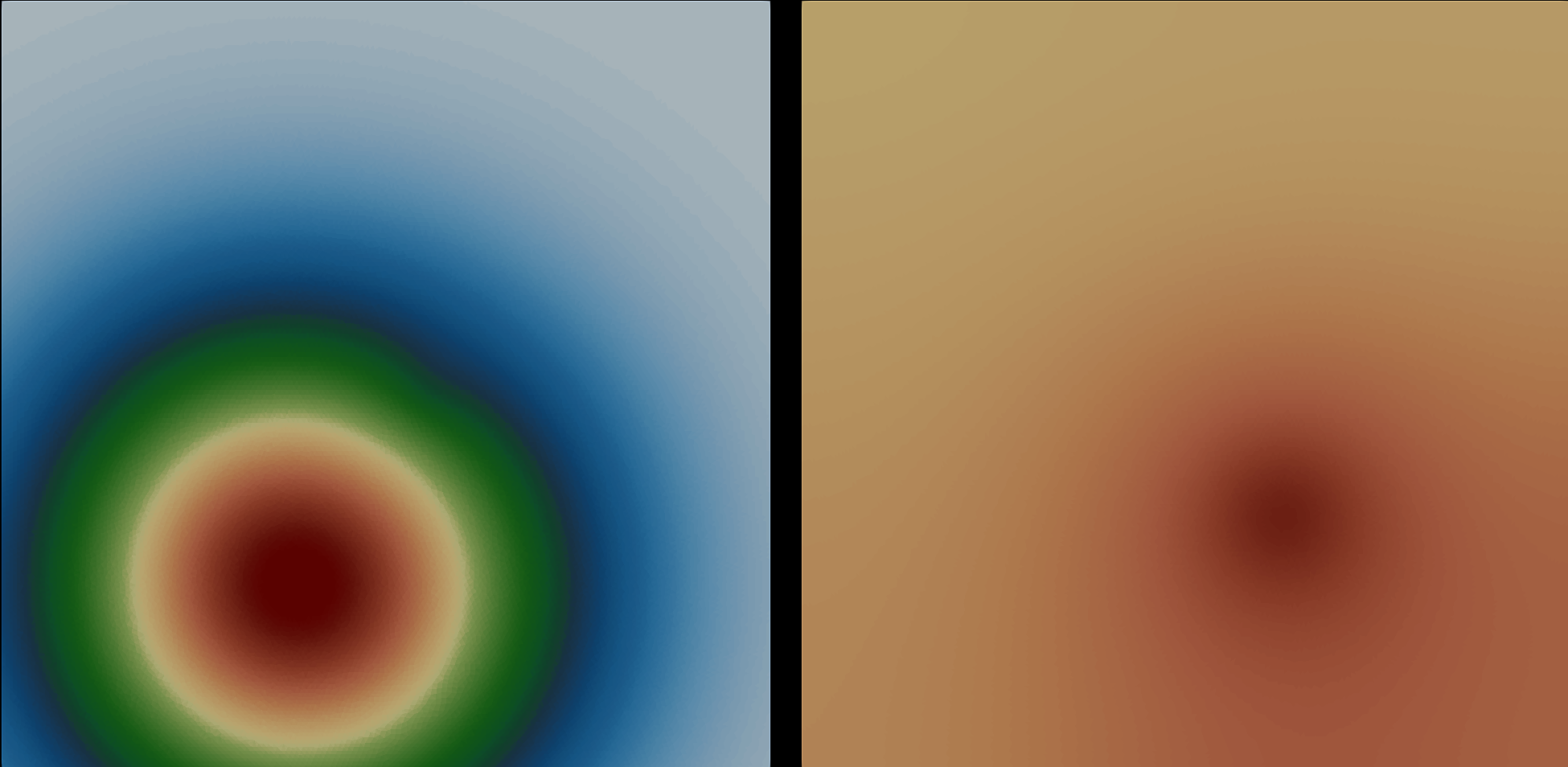}
\end{minipage}
\begin{minipage}[c]{.32\linewidth}
\includegraphics[width=4cm,height=2cm]{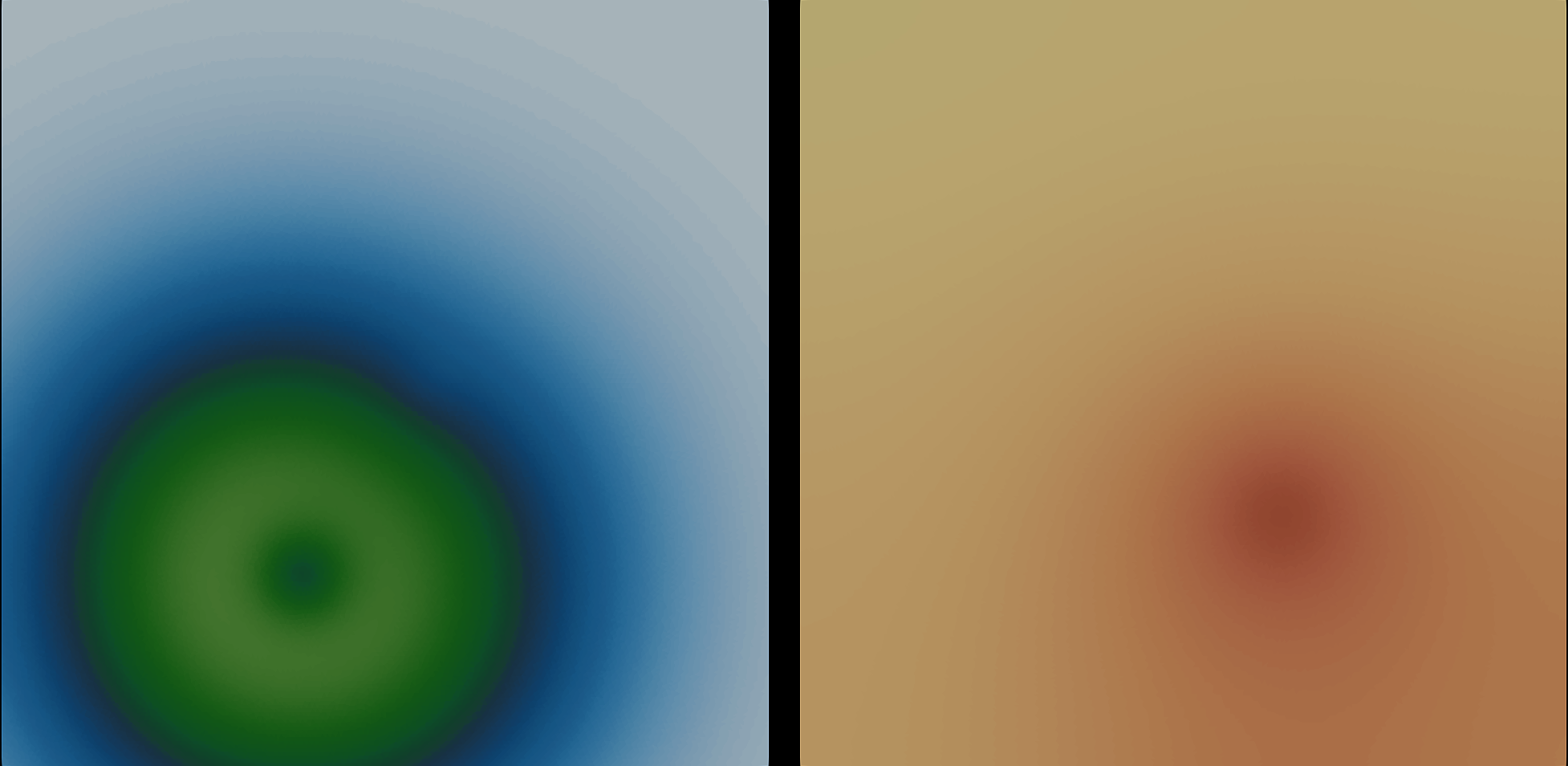}
\end{minipage}
\begin{minipage}[c]{.32\linewidth}
\includegraphics[width=4cm,height=2cm]{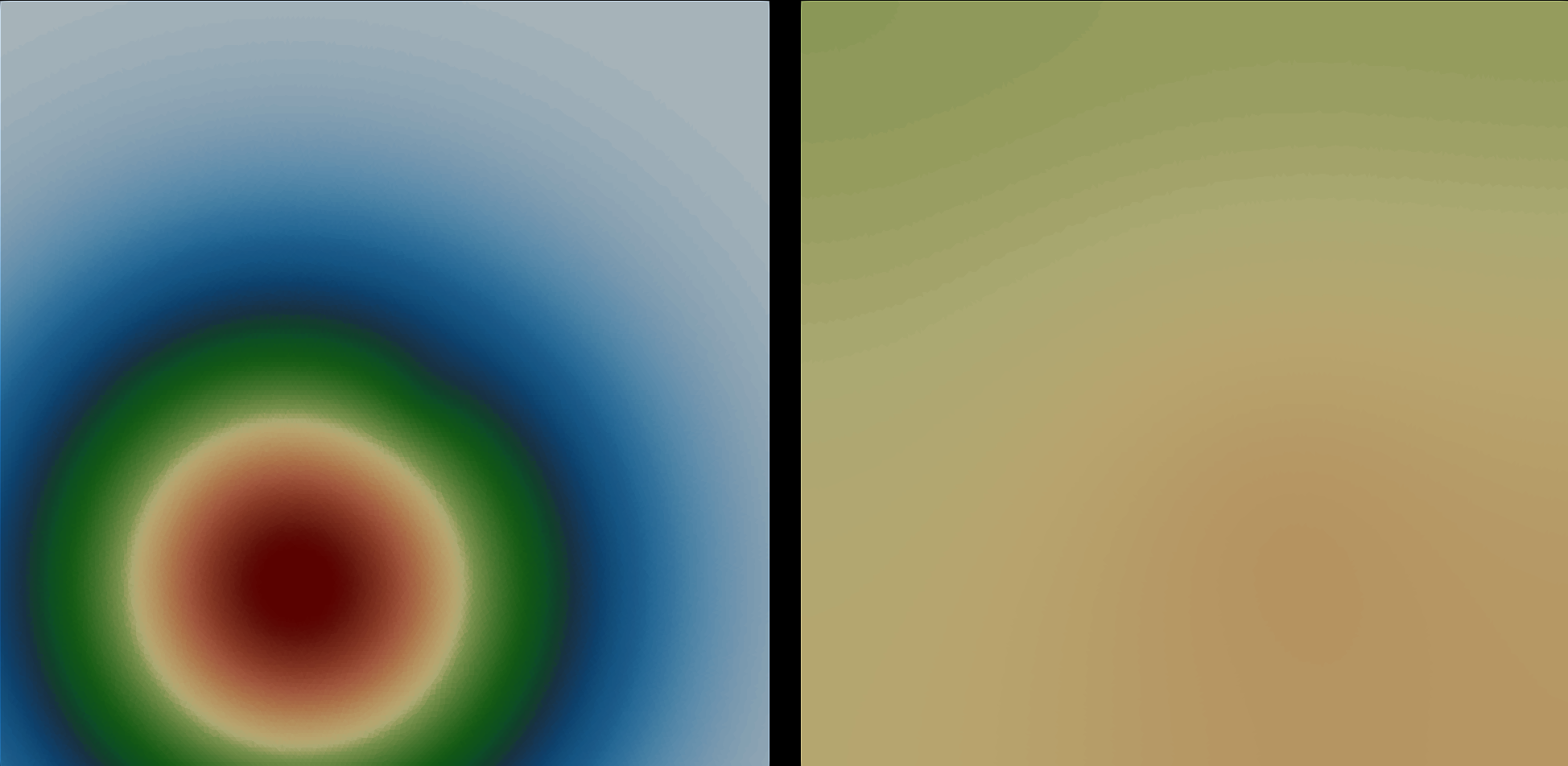}
\end{minipage}
\begin{center}
 $t=400$
\end{center}

\begin{center}
\begin{minipage}[c]{.5\linewidth}
\includegraphics[width=6cm,height=1.5cm]{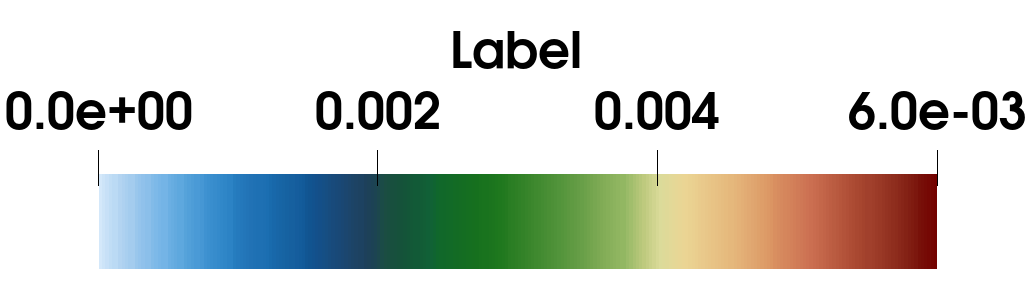}
\end{minipage}
\end{center}
\caption{{\small{This figure illustrates our numerical simulations described in Section \ref{sec:4.1}. Each column corresponds to a different model case. The first column (Wc) shows the evolution of the unidirectional model under consideration. The second column (Sc1) represents the first scenario, where a local control strategy is implemented in the departure area. The third column (Sc2) represents the second scenario, where a local control strategy is applied in the arrival area. The rows represent different time snapshots, highlighting the spatiotemporal evolution of stress within the domains $\Omega_1$ and $\Omega_2$.
}}}
\label{Others behaviors}
\end{figure}
\newpage
\clearpage
\subsection{Discussion}
Our numerical simulations reveal the intricate spatiotemporal dynamics of a population situated within interconnected zones during a dangers situation. These simulations account for diffusion, advection, transition, and imitation within each zone, alongside unidirectional migration from zone $\Omega_1$ to zone $\Omega_2$. Given that the population considered is of a low-risk culture, we focus on the propagation of stress within the two zones and its transfer from $\Omega_1$ to $\Omega_2$. The disaster is assumed to occur in $\Omega_1$, where the entire population is initially stressed, while in $\Omega_2$, the population begins in a non-stressed state. 

In Figure \ref{Others behaviors}, we can observe the spatiotemporal behavior of our unidirectional model, as well as the two control scenarios, allowing us to compare the three model cases: (Wc), (Sc1), and (Sc2). We specify the following observations with respect to each parameter:

\textbf{Diffusion:} Population diffusion occurs randomly within each zone. This can be observed in $\Omega_1$ from $t=0$ to $t=10$ in all scenarios (Wc), (Sc1) and (Sc2), and in $\Omega_2$ as the population spreads from the arrival area throughout the domain, particularly from $t=20$ to $t=400$, see (Wc).

\textbf{Advection:} In $\Omega_1$, from $t=0$ to $t=250$ in the three scenarios (Wc), (Sc1) and (Sc2),  we can see that the stressed population clearly moves towards the target departure area given in Figure \ref{fig:zone-dep-arr}.

\textbf{Transition and Imitation:} These dynamics are evident in $\Omega_2$, where the initial population is entirely non-stressed. The progression of stress levels due to imitation and transition is observable from $t=0$ to $t=400$, especially in the first case (Wc). This is further corroborated by the time evolution diagrams (see Figures \ref{pop-fig1} and \ref{pop-fig2}).

\textbf{First Control Scenario:} Implementing control in the departure area effectively reduces stress levels in $\Omega_1$. The control is initiated at time $T_0=5$ and ends at $T_1=20$. This reduction is evident when comparing the (Wc) scenario from $t=20$ where the control strategy is at the maximum to $t=400$ with the (Sc1) scenario over the same time period in $\Omega_1$, which subsequently allows for reduced stress levels in $\Omega_2$, see $t=400$ in (Wc) and (Sc1).

\textbf{Second Control Scenario:} Implementing control in the arrival area significantly reduces stress levels in $\Omega_2$. The control is initiated at time $T_0=10$ and reaches its maximum effect at $T_1=20$. This reduction is evident when comparing the (Wc) scenario from $t=20$ to $t=400$ with the (Sc2) scenario over the same time period in $\Omega_2$. \\
Moreover, by comparing the three cases (Wc), (Sc1), and (Sc2) at $t=400$, it is clear that the first control strategy is more effective in reducing stress levels in the whole network, while the second control strategy is more effective in reducing the stress levels in the zone $\Omega_2 $, this can be confirmed by the temporal evolution diagrams in Figures  \ref{pop-fig1} and \ref{pop-fig2}.

\section{Conclusion}\label{sec:5}

In this work, we introduced a compartmental advection-diffusion network model to describe the propagation of stress in a population situated in two interconnected spatial zones during a disaster situation. Our model accounts for interactions facilitated by intrinsic transitions and imitations within each zone, as well as migrations between these zones. By employing semigroup theory and abstract evolution equations, we proved the local existence, uniqueness, and regularity of the solutions. Furthermore, we established the positivity and $L^1$--boundedness of the solutions.

Through various numerical simulations, we illustrated different scenarios of stress propagation. We also implemented a local control strategy aimed at minimizing stress levels in a population facing a low-risk culture during a dangerous situation. The results demonstrate the effectiveness of the proposed model and control strategy in managing and mitigating stress propagation in interconnected zones during disaster situations.

Overall, our findings contribute to the understanding of stress dynamics and social behaviors of pedestrians in danger situations and provide a foundation for developing effective control strategies to manage stress levels in affected populations.
\newpage
\clearpage
\begin{appendices}
	\section{Table of numerical values }
	
	\begin{table}[h!]
		\caption{Parameters of the model \eqref{Main Model}}
		\label{tab:parameters-systeme-ACP-0}
		\centering
		\begin{tabular}{l| l| l}
			\textbf{Parameters}                            & \textbf{Zone 1}   &  \textbf{Zone 2}   \\
			
			\hline
			$d_{P_i}$	& $0.2$ & $0.15$ \\
			$d_{N_1}$ & $0.1$ & $0.05$ \\
			$v_{P_i,max}$ & $0.025$ & $0$ \\
			$v_{N_i,max}$ & $0.015$ & $0$ \\
			$a_i$       &        $0.01$    &        $0.005$          \\
			
			$b_i$       &        $0.005$      &         $0.0005$             \\
			$\alpha_{P_i}$  & $0.7$ & $0.5$\\
			
			$\alpha_{N_i}$  & $0.4$ & $0.4$\\
			
			$m_{ji}$  & $0.2$         &      $0.8$   \\
			$K_i$       &        $1$      &         $1$             \\
			
		\end{tabular}
	\end{table}
\end{appendices}

\end{document}